\documentclass{amsart}
\usepackage{amsmath,amsthm,amssymb,comment,mathrsfs}
\usepackage[dvips]{graphicx}
\usepackage{color,xcolor}
\definecolor{shozyohi}{rgb}{0.7686, 0, 0}
\usepackage[dvips,
setpagesize=false,%
bookmarks=true,%
bookmarksnumbered=true,%
colorlinks=true,%
linkcolor=shozyohi,%
anchorcolor=black,%
citecolor=blue,%
filecolor=cyan,%
urlcolor=magenta,%
menucolor=red,%
runcolor=cyan,%
pdftitle={},%
pdfauthor={},%
pdfsubject={},%
pdfkeywords={}]{hyperref}%
\makeatletter
\theoremstyle{plain}
\newtheorem{thm}{Theorem}[section]
\newtheorem*{thm*}{Theorem}
  \newtheorem{prop}[thm]{Proposition}
  \newtheorem{lem}[thm]{Lemma}
  \newtheorem{cor}[thm]{Corollary}
  
\theoremstyle{definition}
  \newtheorem{dfn}[thm]{Definition}
  \newtheorem{exmp}[thm]{Example}
  
  \newtheorem{axm}[thm]{Axiom}
  \newtheorem*{conv}{Convention}
  \newtheorem{rem}[thm]{Remark}
\numberwithin{equation}{section}
\renewcommand\labelenumi{\textrm{(\arabic{enumi})}}
\newcommand*\gl{\lambda}
\newcommand*\gm{\mu}
\newcommand*\gn{\nu}
\newcommand*\gD\Delta
\newcommand*\gL\Lambda
\newcommand*\bm[1]{{\boldsymbol{#1}}}
\newcommand*\ba{\bm{a}}
\newcommand*\bb{\bm{b}}

\newcommand*\zvec{\mathbf{0}}
\newcommand*\kk{\Bbbk}

\newcommand*\NN{\mathbb{N}}
\newcommand*\ZZ{\mathbb{Z}}

\newcommand*\RR{\mathbb{R}}
\newcommand*\cA{\mathcal{A}}
\newcommand*\cB{\mathcal{B}}
\newcommand*\cC{\mathcal{C}}
\newcommand*\cF{\mathcal{F}}
\newcommand*\cH{\mathcal{H}}
\newcommand*\cI{\mathcal{I}}
\newcommand*\cL{\mathcal{L}}
\newcommand*\cM{\mathcal{M}}

\newcommand*\cV{\mathcal{V}}
\newcommand*\fp{\mathfrak{p}}
\newcommand*\sfm{\mathsf{m}}
\newcommand*\sn{\mathsf{n}}
\newcommand*\wS{\widetilde{S}}
\let\opn\operatorname 
\newcommand*\Hom{\opn{Hom}}
\newcommand*\Ext{\opn{Ext}}

\newcommand*\Ass{\opn{Ass}}
\newcommand*\Min{\opn{Min}}
\newcommand*\lcm{\opn{lcm}}
\newcommand*\rank{\opn{rank}}
\newcommand*\chara{\opn{char}}
\newcommand*\supp{\opn{supp}}
\newcommand*\gr{{\opn{gr}}}

\newcommand*\void\varnothing

\newcommand*\signs{\cbr{0,-,+}}
\newcommand*\bpol{\opn{b-pol}}
\newcommand*\Iff{\Longleftrightarrow}
\newcommand*\imply{\Rightarrow}
\newcommand*\longto{\longrightarrow}

\newcommand\pr[1]{\left( #1\right)}
\newcommand\cbr[1]{\left\{ #1\right\}}
\newcommand\br[1]{\left[ #1\right]}

\newcommand\ang[1]{\left\langle #1\right\rangle}
\newcommand\set[2]{\left\{ #1 \mid #2 \right\}}
\newcommand*\ul\underline
\title[The CM-ness of the bounded complex of an affine oriented matroid]%
{The Cohen--Macaulayness of the bounded complex of an affine oriented matroid}\thanks{This work was supported by JSPS KAKENHI Grant Number 24740013, 15K17514, 25400057, 16K05114.}
\author{Ryota Okazaki}
\address{Faculty of Eduation, University of Teacher Education Fukuoka, Munakata, Fukuoka 811-4192, Japan}
\email{rokazaki@fukuoka-edu.ac.jp}
\author{Kohji Yanagawa}
\address{Department of Mathematics, Kansai University, Suita, Osaka 564-8680, Japan}
\email{yanagawa@ipcku.kansai-u.ac.jp}
\begin{document}
\begin{abstract}
An affine oriented matroid $\cM$ is a combinatorial abstraction of an affine hyperplane arrangement. From $\cM$, 
Novik, Postnikov and Sturmfels \cite{NPS} constructed a squarefree monomial ideal $O_\cM$ in a polynomial ring $\wS$, 
and got beautiful results. Developing their theory, we will show the following.
\begin{enumerate}
\item If $\wS/O_\cM$ is Cohen--Macaulay, then the bounded complex $\cB_\cM$ 
(a regular CW complex associated with  $\cM$) is a contractible homology manifold with boundary. 
 This is closely related to Dong's theorem (\cite{D08}), which used to be Zaslavsky's conjecture. 
\item We give a characterization of  $\cM$ such that $\wS/O_\cM$ is Cohen--Macaulay, which states that the converse of \cite[Corollary 2.6]{NPS} is essentially true. 
\end{enumerate}
\end{abstract}
%
%
\maketitle
%
\section{Introduction}
An {\it oriented matroid} is a pair of a finite set $E$ and a set $\cL \subset \signs^E$ of sign vectors satisfying some axioms. It is considered as a common abstraction of many sorts of mathematical objects and deep theory has been constructed (see \cite{BVSWZ}). 
A typical example is the one given by a linear hyperplane arrangement in a Euclidean space. In general, by Topological Representation Theorem (abbrev. TRT), any oriented matroid (without loops) can be realized as an arrangement of ``pseudo-equators'' indexed by the elements of $E$ in a $d$-sphere. For example, an oriented matroid coming from a linear hyperplane arrangement in $\RR^{d+1}$ is represented as an arrangement on the unit sphere in $\RR^{d+1}$.

An affine oriented matroid is just a triad $(E,\cL,g)$ such that $(E,\cL)$ is an oriented matroid and $g \in E$.
Philosophically, $(E,\cL,g)$ corresponds to an arrangement on the open hemisphere with respect to  the ``equator'' $g$,
and an affine hyperplane arrangement $\RR^d$ is a typical example.
An affine hyperplane arrangement $\cA$ in $\RR^d$ decomposes $\RR^d$ into a finite number of cells, and the bounded cells form a regular CW complex, called the {\it bounded complex} of $\cA$.
Similarly, any affine oriented matroid $\cM$ admits the bounded complex $X(\cB_\cM)$ that is also finite regular CW by TRT.

Intuitively, bounded complexes seem to behave well; indeed they are always contractible as shown by Bj\"orner and Ziegler \cite[Theorem 4.5.7]{BVSWZ}. Recently, Dong \cite{D08} has shown that $X(\cB_\cM)$ is homeomorphic to a ball if $\cM$ is \emph{uniform} (see Remark~\ref{rem:genpos-uniform} for the definition).
When $\cM$ comes from an affine hyperplane arrangement, Dong's theorem was first conjectured by Zaslavsky \cite{Z}.
At the same time, there are lots of examples of $\cM$ that is not uniform but whose bounded complex is a ball (see Example~\ref{exmp:cm-arr} and Remark~\ref{rem:CM-imply-CM} for example).
These observations lead us to expect that $X(\cB_\cM)$ still satisfies nice properties for much wider classes.

Our main results are concerned with an affine oriented matroid $\cM := (E,\cL,g)$ such that $g$ being \emph{in general position}.
Note that the condition of $g$ being in general position is quite different from the one that an affine hyperplane arrangement is in general position in the sense of \cite{St07}.
The condition of $g$ being in general position is much weaker than that of $\cM$ being uniform.  
For example, the affine oriented matroid given by the affine hyperplane arrangement $\cA$ in Figure~\ref{fig:pp-cm-arr} is not uniform since $\cA$ has 3 lines intersecting with a point.
On the other hand, $g$ is in general position because any two lines are not parallel. See Remark~\ref{rem:genpos-uniform} and Example~\ref{exmp:aff-hyp-arr} for details.

The following is one of our main results.

\begin{thm*}[cf.\ {Corollary~\ref{cor:X-hmfd}}]
  If $g$ is in general position, then $X(\cB_\cM)$ is a contractible homology manifold with boundary over $\ZZ$.
\end{thm*}

From \cite{MR} and basic facts in PL topology, it is easy to verify that $X(\cB_\cM)$ is homeomorphic to a ball if $g$ is in general position, in the following cases: $\dim X(\cB_\cM) =2$, or $\cM$ comes from an affine hyperplane arrangement in $\RR^3$.

  In the sequel, we set $E \setminus \cbr g = \cbr{1,\dots,n}$.
  In the proof of the above theorem, a key role is played by an ideal and its minimal free resolution constructed by Novik, Postnikov, and Sturmfels in \cite{NPS}.
  In the paper, they associated, with an affine oriented matroid $\cM = (E,\cL,g)$, a squarefree monomial ideal $O_\cM$ of $\wS:=\kk[x_1, \ldots, x_n, y_1, \ldots, y_n]$ over a field $\kk$,
  and proved that $\wS/O_\cM$ has a minimal free resolution supported by $X(\cB_\cM)$.
  For example, let $\cM$ be the affine oriented matroid associated with the line arrangement in Figure~\ref{4lines}.
  Then $\cM$ is uniform, and $O_\cM =(x_1x_2, x_1y_3, x_1x_4, x_2x_3, x_2y_4, x_3x_4)$.
  The bounded complex $X(\cB_\cM)$ is the shaded part of  Figure~\ref{4lines},
  and supports the minimal free resolution $0 \to \wS^3 \to \wS^8 \to \wS^6\to \wS \to 0$ of $\wS/O_\cM$.
  
\begin{figure}[htbp]
\centering
\includegraphics{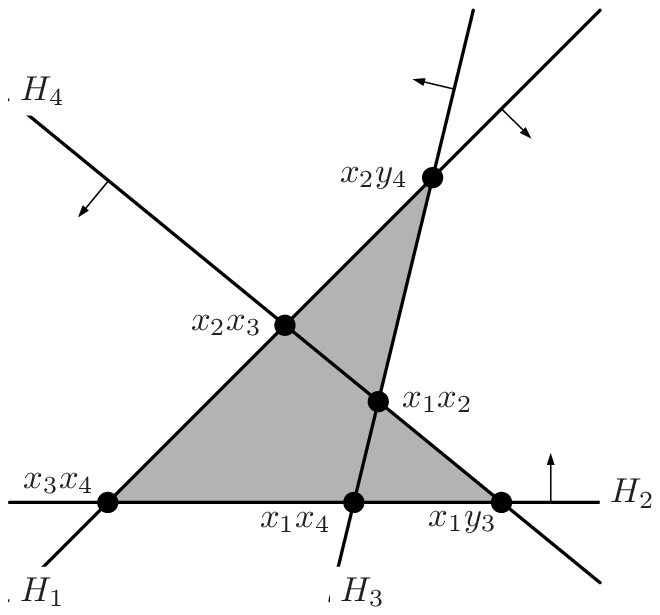}
\caption{The bounded complex $X(\cB_\cM)$} 
\label{4lines}
\centering
\end{figure}

Throughout this paper, we will use some results and techniques in commutative algebra. See \cite{BH,St} for undefined terminology. Novik et al.\ also showed that if $g$ is in general position, then $\wS/O_\cM$ is Cohen--Macaulay. Unfortunately the converse is not true in general, but in Theorem~\ref{thm:char-genpos}, we will show that, under the mild condition that $\cB_\cM$ is of full rank, the following three conditions are equivalent: (1) $g$ is in general position, (2) $\wS/O_\cM$ is Cohen--Macaulay, and (3) by setting $y_i$ to $x_i$, $\wS/O_\cM$ degenerates to a Stanley--Reisner ring of an ordinary matroid. 
It is noteworthy that we do not need any assumption for the equivalence of (2) and (3), and hence the Cohen--Macaulayness of $\wS/O_\cM$ does not depend on the characteristic of the base field $\kk$.

To prove the main theorem above, we will introduce a new notion, which we call a \emph{faithful} cellular resolution.
Additionally, we will give a concrete description of the canonical module of $\wS/O_\cM$ as an ideal of $\wS/O_\cM$, when $\wS/O_\cM$ is Cohen--Macaulay.

Since $O_\cM$ is squarefree, there is the unique simplicial complex $\gD_\cM$ whose Stanley--Reisner ring is $\wS/O_\cM$. We will also study $\gD_\cM$ and prove the following.

\begin{thm*}[cf.\ {Theorem~\ref{Delta_M ball}}]
   If $\wS/O_\cM$ is Cohen--Macaulay, then the geometric realization of $\gD_\cM$ is a homology manifold with boundary. Moreover, the boundary is a homology sphere in the sense of \eqref{homology sphere}. 
\end{thm*}
\section{Preliminary}\label{sec:prelim}
In this paper, we mainly treat an affine oriented matroid. 
To introduce it, let us first recall definitions and some basic notions on oriented matroids that we will need in this paper. 
See \cite{BVSWZ} for details and unexplained terminology. 
%
%

Let $E$ be a finite set.
A {\em sign vector} on $E$ is a function $\gl:E \to \signs$. 
Let us remark that a sign vector is usually defined as a vector over $\signs$ indexed by $E$ (with some fixed linear ordering), but we adopt the definition above just for our convenience. 
The set of all the sign vectors on $E$ is denoted by $\signs^E$. 
For a sign vector $\gl \in \signs^E$, we set $\supp(\gl) := \gl^{-1}(\cbr{-,+})$ and call it the \emph{support} of $\gl$. 
The symbol $\zvec$ denotes the sign vector with the empty support.
For convenience, we consider the natural multiplication on $\signs$ as follows:
\[
++ = -- =+,\ +-=-+ =-,\ +0 = 0+ = -0 = 0-= 0.
\]
The {\em opposite} $-\gl$ of $\gl$ is the sign vector given by $(-\gl)(e) = -\gl(e)$ for all $e \in E$. 
Defining $0 <+$, $0 <-$ with $+$ and $-$ incomparable, we have a partial order $\le$ on $\signs$. 
By abuse of notation, let $\le$ denote the partial order on $\signs^E$ defined as follows: 
\[
\gl \le \gm \quad \Iff \quad \gl(e) \le \gm(e) \quad \text{for all $e \in E$}.
\]
With this order, the sign vector $\zvec$ is the least in $\signs^E$.
For two sign vectors $\gl,\gm$, their {\em composition} $\gl \circ \gm$ is the sign vector defined as follows:
\[
(\gl \circ \gm)(e) := \begin{cases}
\gm(e) & \text{for }e \not\in\supp(\gl), \\
\gl(e) & \text{for }e \in \supp(\gl).
\end{cases}
\]
It is clear that $\gl \circ \gm \ge \gl$ for all $\gl, \gm \in \cL$.
For $e \in E$ and $\gl, \gm \in \signs^E$, it follows from the definition that $(\gl \circ \gm)(e) = (\gm \circ \gl)(e)$ if and only if $e$ does not belong to the set
\[
S(\gl,\gm) := \set{f \in E}{\gl(f) = - \gm(f) \neq 0},
\]
called the {\em separation set} of $\gl$ and $\gm$. 

Now we can state the definition of an oriented matroid in terms of covectors. 
A pair $\cM = (E,\cL)$ of a finite set $E$ and a subset $\cL \subseteq \signs^E$ is said to be an \emph{oriented matroid} (on $E$) if the following axioms, called \emph{Covector Axioms}, are satisfied.
\begingroup%
\begin{axm}[Covector Axioms]\ 
\renewcommand\labelenumi{(L\arabic{enumi})}%
\begin{enumerate}%
\setcounter{enumi}{-1}
\item $\zvec \in \cL$,
\item $\gl \in \cL$ implies $-\gl \in \cL$,
\item $\gl \circ \gm \in \cL$ for any $\gl,\gm \in \cL$, and
\item for $\gl,\gm \in \cL$ and $e \in S(\gl,\gm)$, there exists $\gn \in \cL$ such that $\gn(e) = 0$ and $\gn(f) = \gl \circ \gm(f) = \gm \circ \gl(f)$ for all $f \in E \setminus S(\gl,\gm)$.
\end{enumerate}%
\end{axm}
\endgroup

\begin{exmp}[linear hyperplane arrangement] \label{exmp:lin-hyp-arr}
Let $E$ be a finite set and $\set{L_e}{e \in E}$ a linear hyperplane arrangement in $\RR^d$. For $e \in E$, let $v_e \in \RR^d$ be the vector defining $L_e$. 
Each vector $v$ in $\RR^d$ defines the sign vector $\gl_v \in \signs^E$ as follows:
\begin{align} \label{eq:vec-fct}
\gl_v(e) := \begin{cases}
+ &\text{if }\ang{v_e,v} > 0, \\
0 &\text{if }\ang{v_e,v} = 0, \\
- &\text{if }\ang{v_e,v} < 0,
\end{cases}
\end{align}
where $\ang{-,-}$ denotes the inner product in $\RR^d$.
The pair $(E,\set{\gl_v}{v \in \RR^d})$ is then an oriented matroid.
\end{exmp}

A \emph{loop} of $\cM$ is an element $e \in E$ such that $\gl(e) = 0$ for any covector $\gl$, and a \emph{coloop} of $\cM$ is such that $\supp(\gl) = \cbr{e}$ for some covector $\gl$.
For a subset $\cV \subseteq \signs^E$, let $\Min\cV$ denote the set of the support-inclusion minimal elements in $\cV$. For an oriented matroid $\cM := (E, \cL)$, the elements of the set $\cC:= \Min(\cL\setminus\cbr\zvec)$ are called \emph{cocircuits} of $\cM$. It is an easy exercise to show that
the cocircuit set $\cC$ coincides with the set $\Min_\le (\cL \setminus \cbr \zvec)$ of the elements of $\cL \setminus \cbr \zvec$ that are minimal with respect to $\le$.
The set of cocircuits is characterized by the axiom called (\emph{Co})\emph{circuit Axioms}. See \cite{BVSWZ} for details.

Note that the sets
\[
\underline\cC := \set{\supp(\gl) \subseteq \br n}{\gl \in \cC}
\]
and $E$ form the ordinary matroid $(E,\underline\cC)$ (in terms of Circuit Axiom). The dual matroid of $(E, \underline \cC)$ is called the \emph{underlying matroid} and denoted by $\underline\cM$. The \emph{rank} of $\cM$ is defined to be that of $\underline\cM$. It follows from \cite[Theorem 4.1.14]{BVSWZ} that the rank of $\cM$ is equal to that of $\cL$ as a poset.

\begin{rem}
Following \cite{D,NPS} and \cite[Section 4]{BVSWZ}, we will use terminologies on oriented matroids in the dual form (for oriented matroids, their dual can be defined. See \cite{BVSWZ} for details). The reader should note some differences in notation and concepts when he/she refers to \cite[Section 3]{BVSWZ}.

In the view of the above, the cocircuit set of $\cM$ should be denoted by $\cC^*$; nevertheless we prefer to use $\cC$.
\end{rem}

One of the most remarkable results in oriented matroid theory is the following shellability and sphericity theorem. 

\begin{thm}[cf.\ {\cite[Theorem 4.3.3 and Proposition 4.7.24]{BVSWZ}}] \label{thm:TRT} Let $(E,\cL)$ be an oriented matroid of rank $r$. The poset $\cL$ is isomorphic to the face poset of a shellable regular CW complex $X(\cL)$ whose underlying space is a $(r-1)$-sphere.
\end{thm}

For a sign vector $\gl \in \signs^E$ and a subset $F \subseteq E$, let $\gl|_F$ be the \emph{restriction} of $\gl$ to $F$, that is, the sign vector on $F$ with $\gl|_F(e) = \gl(e)$ for all $e \in F$. 
For an oriented matroid $\cM = (E,\cL)$ and its covector set $\cC$, we set
\[
\cC|_F := \Min\set{\gl|_F}{\gl \in \cC,\ F \cap \supp(\gl) \neq \void}, \quad \cL|_F:= \set{\gl|_F}{\gl \in \cL}.
\]
The pair $\cM|_F := \pr{F,\cL|_F}$ is then an oriented matroid on $F$ and called the \emph{restriction} of $\cM$ to $F$. Note that the set of cocircuits of $\cM|_F$ is equal to $\cC|_F$.

An element $e \in E$ is said to be \emph{in general position} (\cite[Proposition 7.2.2 (2)]{BVSWZ}) if $e$ is not a coloop and $\cC|_{E'} \subseteq \set{\gl|_{E'}}{\gl \in \cC,\ e \in \supp(\gl)}$, where $E' := E \setminus \cbr e$.

\begin{rem}\label{rem:genpos-uniform}%
  For an oriented matroid $\cM = (E,\cL)$, under the condition that $E$ admits a non-coloop element, every $e \in E$ is in general position if and only if $\cM$ is uniform, or equivalently $\underline\cC = \set{F \subseteq E}{\#F = s}$ for some $s \in \NN$.
\end{rem}

An affine oriented matroid is a triple $\cM := (E,\cL,g)$ consisting of a finite set $E$, a set $\cL$ of some sign vectors, and the distinguished element $g \in E$ such that $(E,\cL)$ is an oriented matroid and $g$ is not a loop. 
The \emph{positive part} $\cL^+$ of $\cM$ and the \emph{bounded complex} $\cB_\cM$ of $\cM$ are, by definition,
\[
\cL^+ := \set{\gl \in \cL}{\gl(g) = +}, \quad
\cB_\cM := \set{\gl \in \cL^+}{(\zvec, \gl] \subseteq \cL^+},
\]
where $(\zvec, \gl] := \set{\gm \in \cL}{\zvec < \gm \le \gl}$.

\begin{conv}
  Let $n \in \NN$ and set $\br n := \cbr{1,\dots,n}$. Throughout this paper, we prefer to set the base set $E$ of affine oriented matroids to be $\br n \cup \cbr g$ with $g \not\in \br n$. Moreover unless otherwise stated, we tacitly assume that the distinguished element $g$ of $\cM = (\br n \cup \cbr g, \cL, g)$ is not a coloop; hence it follows that $\rank \cM = \rank \cM|_{\br n}$.
\end{conv}

For a CW complex $X$, we write the set of the $i$-cells as $X^{(i)}$ and  set $X^{(*)} := \bigcup_i X^{(i)}$.
The empty set $\void$ is considered as the $-1$-cell. The set $X^{(\ast)}$ forms a partially ordered set by $\sigma \ge \tau$ $\stackrel{\rm def}\Longleftrightarrow$ $\overline{\sigma} \supset \tau$, where $\overline\void := \void$. Henceforth we refer to $X^{(*)}$ also as a CW complex. Recall that $X$ is said to be regular if the characteristic map of each $\sigma \in X^{(\ast)}$ is homeomorphism. Hence the closure of each $\sigma$ is homeomorphic to a closed ball, when $X$ is regular. See \cite{BVSWZ,Sp} for the definition and basic properties of a (regular) CW complex.

 
By Theorem~\ref{thm:TRT}, $\cB_\cM \cup \cbr{\zvec}$ is isomorphic to the face poset of a regular CW complex that is a subcomplex of $X(\cL)$. We write $X(\cB_\cM)$ to denote this complex and call it the \emph{bounded complex} of $\cM$ as the same with $\cB_\cM$ by abuse of terminology.

A typical example of $\cB_\cM$ and $X(\cB_\cM)$ is the bounded complex of an affine hyperplane arrangement.

\begin{exmp}[affine hyperplane arrangement]\label{exmp:aff-hyp-arr}
  Let $\cbr{L_1,\dots ,L_n,L_g}$ and $v_1,\dots ,v_n,v_g$ be a linear hyperplane arrangement in $\RR^{d+1}$ and vectors defining $L_i$'s. Set $H_g := \set{v \in \RR^{d+1}}{\ang{v,v_g} = 1}$ and $H_i := L_i \cap H_g$ for $i = 1,\dots ,n$. The set $\cA := \cbr{H_1,\dots ,H_n}$ is then an affine hyperplane arrangement in $\RR^d\cong H_g$.
  Any affine hyperplane arrangement in $\RR^d$ is constructed in this way. Letting $\cL$ be the covector set given by $\cbr{L_1,\dots ,L_n,L_g}$ as Example~\ref{exmp:lin-hyp-arr}, we obtain the affine oriented matroid $\cM := (\br n \cup \cbr g,\cL,g)$. The positive part $\cL^+$ is then equal to $\set{\gl_v}{v \in H_g}$. Hence $\cA$ corresponds to $\cL^+$ in this sense. Moreover $\cB_\cM$ corresponds to the bounded regions given by $\cA$, whenever $V := \bigcap_{i \in \br n \cup \cbr g}L_i = 0$, or equivalently $\rank\cM = d + 1$. See Examples \ref{exmp:cm-arr} and \ref{exmp:nongp-cm-arr} for concrete examples. Clearly, $g$ is coloop if and only if $\cA$ is central (i.e., $\bigcap_{i=1}^nH_i \neq \void$).

  Assume $V = 0$. Then $g$ is in general position if and only if $\dim_\RR L_A \cap L_g = 0$ for all $A \subseteq \br n$ with $\dim_\RR L_A = 1$, where $L_A := \bigcap_{i \in A}L_i$. When $\cA$ is a line arrangement in $\RR^2$, this is equivalent to say that no two distinct lines in $\cA$ are parallel. 
\end{exmp}

Let $S := \kk\br{x_1,\dots ,x_n}$ be the polynomial ring over a field $\kk$ with indeterminates $x_1,\dots ,x_n$, and set $\widetilde S := S \otimes_\kk \kk\br{y_1,\dots ,y_n}$, where $y_1,\dots ,y_n$ are variables. For a sign vector $\gl \in \signs^{\br n \cup \cbr g}$, define
\[
\sfm_\gl := \prod_{i \in \gl^{-1}(+) \setminus \cbr g}x_i \cdot \prod_{i \in \gl^{-1}(-) \setminus \cbr g}y_i \in \widetilde S,
\]
where we set $\sfm_\zvec = 1$. For an affine oriented matroid $\cM := (\br n \cup \cbr{g},\cL,g)$, the ideal
\[
O_\cM := \pr{\sfm_\gl \mid \gl \in \cB_\cM} = \pr{\sfm_\gl \mid \gl \in \cC \cap \cL^+}.
\]
of $\widetilde S$ is called the \emph{matroid ideal} of $\cM$. Note that $g$ is a coloop if and only if $O_\cM = \widetilde S$.

\begin{exmp} \label{exmp:cm-arr}
(1) Let $v_1 := (0,1,0)$, $v_2 := (-1,1,0)$, $v_3 := (1,0,0)$, $v_4 := (1,1,-1)$, $v_g := (0,0,1)$ be the vectors in $\RR^3$, and $L_1,L_2,L_3,L_4, L_g$ be the linear hyperplanes in $\RR^3$ defined by these vectors, respectively. Define $H_i$ ($i \in \br 4 \cup \cbr g$) and $\cM := (\br 4 \cup \cbr g,\cL,g)$ as Example~\ref{exmp:aff-hyp-arr}. 
  Figure~\ref{fig:pp-cm-arr} indicates the affine hyperplane arrangement $\cA := \set{H_i}{i \in \br 4}$ and Figure~\ref{fig:zp-cm-arr} the section by $L_g$.
  The simplicial complex displayed in the shaded area and its face poset correspond to $X(\cB_\cM)$ and $\cB_\cM$ respectively. For this example, $g$ is in general position. The vertices in the figure correspond with the cocircuits $\gl$ of $\cM$ with $\gl(g) = +$, and hence $O_\cM = (x_1x_2,x_1x_3,y_2x_3,y_4)$.

  \begin{figure}[htbp]
  \begin{tabular}{cc}
    \begin{minipage}[b]{.49\textwidth}
      \centering
      \includegraphics{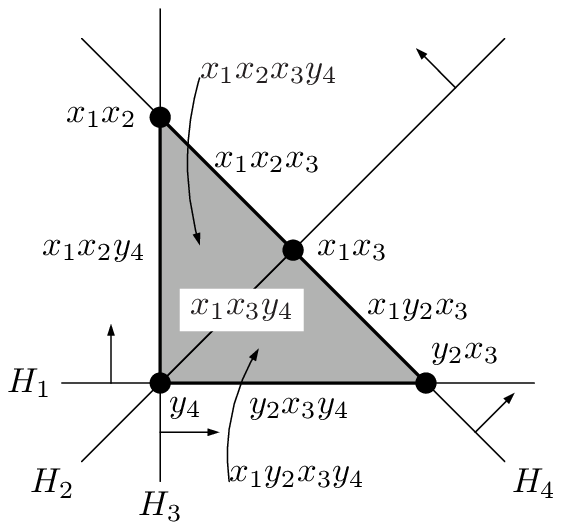}
      \caption{$X(\cB_\cM)$} \label{fig:pp-cm-arr}
    \end{minipage}
    &\begin{minipage}[b]{.49\textwidth}
      \centering
      \includegraphics{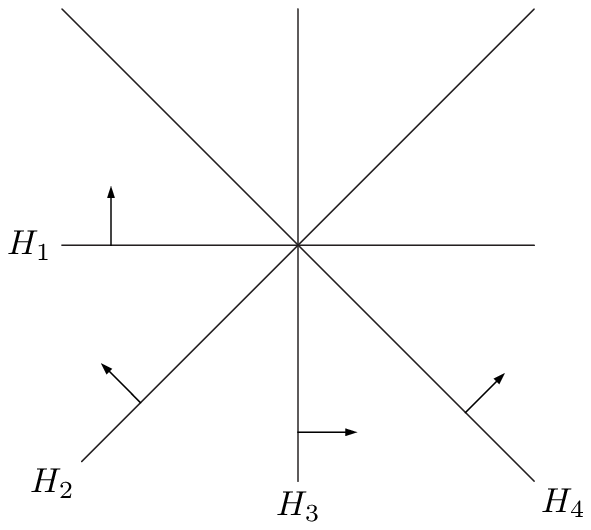}
      \caption{Section by $L_g$} \label{fig:zp-cm-arr}
    \end{minipage}
  \end{tabular}
\end{figure}

(2) Let $v_1 := (1,0,1)$, $v_2 := (1,0,-1)$, $v_3 := (-1,1,0)$, $v_4 := (1,1,0)$, $v_g := (0,0,1)$.
The affine hyperplane arrangement $\cA$ and the section by $L_g$ are then as in Figures~\ref{fig:pp-noncm-arr} and \ref{fig:zp-noncm-arr}.
It thus follows that $g$ is not in general position and $O_\cM = (x_1x_4,x_1y_2,x_1y_3,y_2x_3,y_2y_4)$. 
\begin{figure}[htbp]
  \begin{tabular}{cc}
    \begin{minipage}[b]{.49\textwidth}
      \centering
      \includegraphics{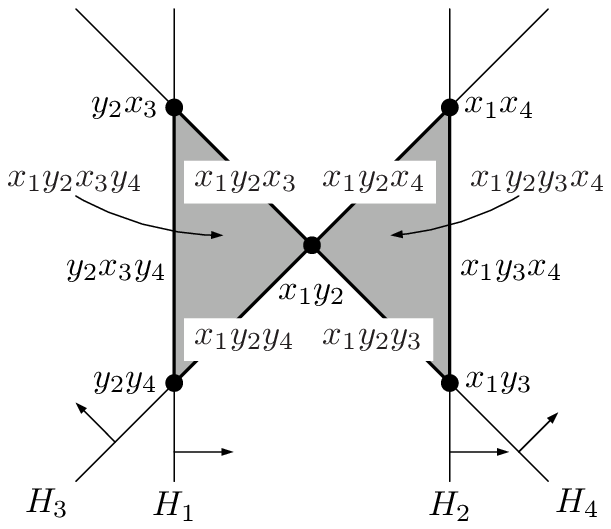}
      \caption{$X(\cB_\cM)$} \label{fig:pp-noncm-arr}
    \end{minipage}
    &\begin{minipage}[b]{.49\textwidth}
      \centering
      \includegraphics{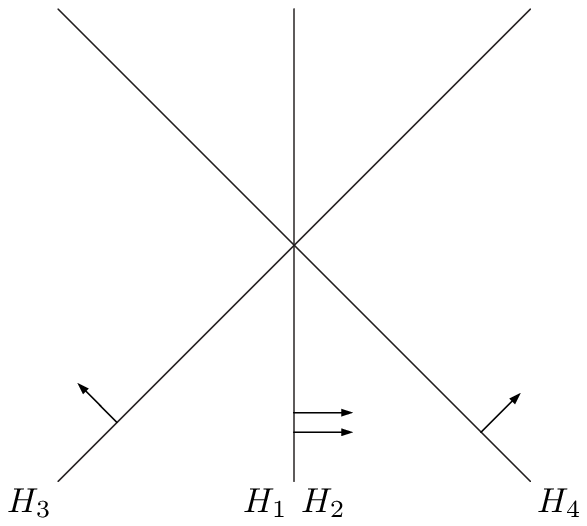}
      \caption{Section by $L_g$} \label{fig:zp-noncm-arr}
    \end{minipage}
  \end{tabular}
\end{figure}
\end{exmp}
%
%
%
%

In their paper \cite{NPS}, I. Novik, A. Postnikov, and B. Sturmfels showed that $O_\cM$ has a minimal cellular $\widetilde S$-free resolution supported by the regular CW complex $X(\cB_\cM)$ (see Section~\ref{sec:CM-implies-CM} for the definition of cellular resolutions).

\begin{thm}[Novik--Postnikov--Sturmfels {\cite[Theorem 2.2]{NPS}}] \label{thm:NPS-cell-res}
Let $\cB_\cM$ be the bounded complex of an affine oriented matroid $\cM := (\br n \cup \cbr g, \cL, g)$, and let $\gr: X(\cB_\cM)^{(*)} \longto \NN^{2n}$ be the function defined as $\gr(\gl) := \deg(\sfm_\gl)$. The pair $(X(\cB_\cM)^{(*)},\gr)$ then gives a minimal cellular resolution of $\widetilde S/O_\cM$ supported by $X(\cB_\cM)^{(*)}$.
\end{thm}

Besides Theorem~\ref{thm:NPS-cell-res}, Novik et al.\ gave the following sufficient condition for Cohen--Macaulayness of $\widetilde S/O_\cM$.

\begin{thm}[Novik--Postonikov--Sturmfels {\cite[Corollary 2.6]{NPS}}] \label{thm:gen-pos}
  Let $\cM := (\br n \cup \cbr g, \cL ,g)$ be an affine oriented matroid of rank $r$ and assume $g$ is neither a loop nor a coloop. 
  If $g$ is in general position, then $\widetilde S/O_\cM$ is Cohen--Macaulay of dimension of $2n - r$.
\end{thm}
%
%

Unfortunately, the converse of Theorem~\ref{thm:gen-pos} is not true in general.

\begin{exmp} \label{exmp:nongp-cm-arr}
Let $v_1 := (1,0,1)$, $v_2 := (-1,0,1)$, $v_3 := (0,1,0)$, $v_g := (0,0,1)$ be the vectors in $\RR^3$ and $L_1,L_2,L_3,L_g$ be the linear hyperplanes defined by $v_1,v_2,v_3,v_g$, respectively. As Figures~\ref{fig:pp-nongp-cm-arr} and~\ref{fig:zp-nongp-cm-arr} show, $O_\cM = (x_1,x_2)$ and $g$ is not in general position. On the other hand, $O_\cM$ is clearly Cohen--Macaulay.
\begin{figure}[htbp]
\begin{minipage}{.49\textwidth}
\begin{center}
\includegraphics{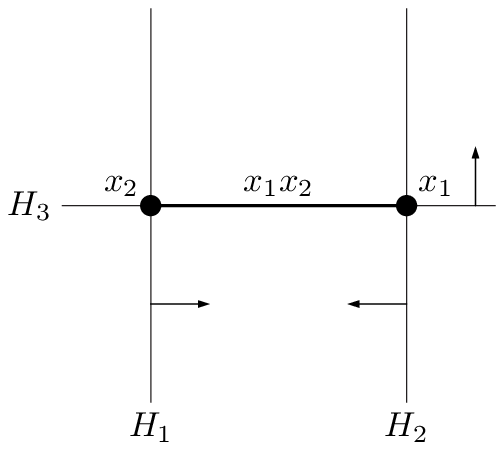}
\caption{$\cB_\cM$} \label{fig:pp-nongp-cm-arr}
\end{center}
\end{minipage}
\begin{minipage}{.49\textwidth}
\begin{center}
\includegraphics{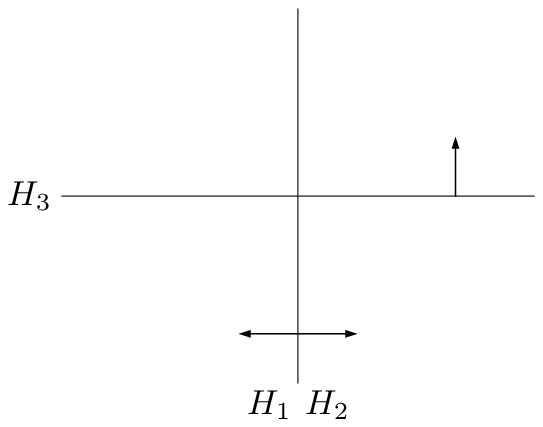}
\caption{Section by $L_g$} \label{fig:zp-nongp-cm-arr}
\end{center}
\end{minipage}
\end{figure}
\end{exmp}

\section{$\widetilde S/O_\cM$ is Cohen--Macaulay implies $\cB_\cM$ is Cohen--Macaulay} \label{sec:CM-implies-CM}
The goal of this section is to show that if $\widetilde S/O_\cM$ is Cohen--Macaulay then $X(\cB_\cM)$ is Cohen--Macaulay as a topological space. To do this, we recall and develop the theory of cellular free resolutions of monomial ideals. See \cite{BS} for details of the theory.

Recall that $S=\kk[x_1, \ldots, x_n]$ is a polynomial ring. 
In the rest of  this section, let  $X^{(\ast)}$ be  a finite regular CW complex. 
Let $\gr :X^{(\ast)} \to \NN^n$ be an order preserving map with $\gr(\void)=\zvec$. Here $\NN^n$  is ordered 
 by componentwise comparison.  Now we define the $\ZZ^n$-graded chain complex $\cF_\bullet^X$  of free $S$-modules as follows: 
set 
$$\cF^X_i := \bigoplus_{\sigma \in X^{(i-1)}} S(-\gr(\sigma)) \, e_\sigma,$$ 
where $e_\sigma$ is an $S$-free
basis, and define the differential map $\partial_i : \cF^X_i \to \cF^X_{i-1}$ by
$$
e_\sigma \longmapsto  \sum_{\tau \in X^{(i-2)}} [\sigma:\tau] \cdot x^{\gr(\sigma) - \gr(\tau)} \cdot e_\tau, $$
where we set $x^{\ba} = \prod_{i=1}^nx_i^{a_i} \in S$ for $\ba = (a_1,\dots ,a_n) \in \NN^n$ and $[\sigma:\tau] \in \ZZ$ denotes the coefficient of $\tau$ in the image
of $\sigma$ by the differential map in the cellular homology of $X^{(\ast)}$. In other words, $[-:-]$ is an {\it incidence function} of the regular CW complex $X^{(\ast)}$. 
In this situation,   we have $\sigma \gtrdot \tau$, if and only if $[\sigma:\tau]  \ne 0$, if and only if  $[\sigma:\tau] = \pm 1$. Here $\sigma \gtrdot \tau$ means that $\sigma$ \emph{covers} $\tau$, i.e., $\sigma > \tau$ and there is no $\upsilon \in X^{(*)}$ between $\sigma$ and $\tau$.
Clearly, $H_0( \cF^X_\bullet) \cong S/I$, where $I$ is the monomial ideal generated by $\{ x^{\gr(v)} \mid v \in X^{(0)}\}$. If $\cF^X_\bullet$ is acyclic, then it gives a free resolution of $S/I$. 
In this case, we call $\cF^X_\bullet$ a {\it cellular resolution} of $S/I$ supported by $X^{(\ast)}$. 

\begin{dfn}\label{faithful}
With the above situation, we say a cellular resolution $\cF_\bullet^X$ is {\it faithful}, if the following are satisfied 
\begin{enumerate}
\item $\gr$ is injective, 
\item $\gr (\sigma) > \gr(\tau)$ implies $\sigma > \tau$. 
\end{enumerate}
\end{dfn}

It is easy to see that if a cellular resolution $\cF_\bullet^X$ is faithful then it is a minimal free resolution.
The Novik--Postnikov--Sturmfels resolution $\cF_\bullet^{X(\cB_\cM)}$ of  $\widetilde{S}/O_\cM$ given  by $\gr: \gl \mapsto \deg(\sfm_\gl)$ (see Theorem~\ref{thm:NPS-cell-res}) is clearly faithful. 

We say a subset $Y$ of $X^{(\ast )}$ is an {\it order filter}, if $\sigma \in X^{(\ast )}$, $\tau \in Y$ and $\sigma \ge \tau$ 
imply $\sigma \in Y$. Typical examples of order filters are the following. 
\begin{enumerate}
\item For $\sigma \in X^{(\ast )}$, $X^{\ge \sigma} := \{\,  \tau  \in X^{(\ast)}\mid \tau \ge \sigma \, \}$ is an order filter. 
\item For an order preserving  map $\gr: X^{(\ast )} \to \NN^n$ and $\ba \in \ZZ^n$, $X^{\ge \ba} := \{\,  \sigma \in X^{(\ast)} \mid \gr(\sigma) \ge \ba \, \}$ is an order filter. 
\end{enumerate}

For an order filter $Y$ of $X^{(\ast )}$, consider the cochain complex  
$C^\bullet(Y)$ of $\kk$-vector spaces with 
$$C^i(Y):= \bigoplus_{\sigma \in Y  \cap X^{(i)}} \kk \, \sigma$$
(here we regard $\sigma \in X^{(\ast)}$ as a basis element)
 and the differential map $\partial: C^i(Y) \to  C^{i+1}(Y)$ is given by 
$$\partial(\sigma)= \sum_{\tau \in Y  \cap X^{(i+1)}} [\tau : \sigma] \, \tau.$$

\begin{prop}\label{ring CM}
Let $I$ be a monomial ideal. Assume that $S/I$ admits a cellular minimal $S$-free resolution $\cF_\bullet^X$.
Then $S/I$ is Cohen--Macaulay if and only if $H^i(C^\bullet(X^{\ge \ba}))=0$ for all $\ba \in \ZZ^n$ and all $i \ne \dim X$.   
\end{prop}

\begin{proof}
It is well-known that  $S/I$ is Cohen--Macaulay if and only if  $\Ext_S^i(S/I,S)=0$ for all $i\ne \opn{proj-dim}_S S/I$. 
Since $\cF_\bullet^X$ is a minimal free resolution of $S/I$, we have 
$H^i(\Hom_S^\bullet(\cF_\bullet^X, S))\cong \Ext_S^i(S/I,S)$ and $\opn{proj-dim}_S S/I = \dim X + 1$. 
Hence  $S/I$ is Cohen--Macaulay if and only if  
$H^i(\Hom_S^\bullet(\cF_\bullet^X, S))=0$ for all $i \ne  \dim X + 1$. 
Since 
$$\Hom_S^i(\cF_\bullet^X, S) = \Hom_S(\cF_i^X, S) \cong \bigoplus_{\tau \in X^{(i-1)}} S(\gr(\tau)),$$
we have 
$$[\Hom_S^i(\cF_\bullet^X, S)]_{-\ba} \cong \bigoplus_{\substack{\tau \in X^{(i-1)} \\  \gr(\tau) \ge \ba}} \kk \, \tau$$
for all $\ba \in \ZZ^n$. By the construction of the differential maps of $\cF_\bullet^X$, we have  
$$[\Hom_S^\bullet(\cF_\bullet^X, S)]_{-\ba} \cong C^{\bullet-1}(X^{\ge \ba})$$ 
and hence 
$$[\Ext_S^i(S/I, S)]_{-\ba} \cong H^{i-1}(C^\bullet(X^{\ge \ba})).$$ So we are done. 
\end{proof}

Note that in the case where $X^{(\ast)}$ is a simplicial complex, for $\sigma \in X^{(*)}$, the complex $C^\bullet(X^{\ge \sigma})$ coincides with the complex given by shifting the augmented cochain complex of the link of $\sigma$, and hence its cohomologies are isomorphic to $H^i(X,X\setminus\cbr{x};\kk)$ for any $x \in \sigma$ whenever $\sigma \neq \void$ (cf.\ \cite[Lemma 63.1]{Mu}). This fact can be generalized to a finite regular CW complex as follows.

\begin{prop} \label{prop:link-iso}
  Let $X^{(\ast)}$ be a finite regular CW complex. For $\sigma \in X^{(*)}$, it follows that
  \[
  H^i(C^\bullet(X^{\ge\sigma})) \cong \begin{cases}
    \widetilde H^i(X;\kk) & \text{if }\sigma = \void, \\
    H^i(X,X\setminus\cbr{x};\kk) &\text{for any } x \in \sigma \text{ if }\sigma \neq \void
    \end{cases}
  \]
  for all $i$.
\end{prop}
\begin{proof}
  The case where $\sigma = \void$ is clear. Assume $\sigma \neq \void$ and $x \in \sigma$.
  Let $A$ be the subspace of $X$ corresponding to the subcomplex $\set{\tau \in X}{\tau \not\ge \sigma}$. 
  Clearly $x \not\in A$ and
\[
H^i(X,A;\kk) \cong H^i(C^{\bullet}(X^{\ge\sigma}))
\]
for all $i$.
  Hence it suffices to show that the maps $H_i(X,A,\kk) \to H_i(X,X\setminus\cbr x;\kk)$ induced from the inclusions $A \subseteq X \setminus\cbr x \subseteq X$ are isomorphisms for all $i$.
  Henceforth every isomorphism is the one induced from inclusion.
  By a standard argument with the universal coefficient theorem and the five lemma (cf.\ \cite[Corollary 5.3.15]{Sp}), the assertion above holds true if $H_i(X,A,\ZZ) \cong H_i(X,X\setminus\cbr x;\ZZ)$ for all $i$.
  As is well-known (cf.\ \cite[Lemma 6.8.6]{Sp}), the long exact sequence induced from $A \subseteq X \setminus\cbr{x} \subseteq X$ implies that $H_i(X,A;\ZZ) \cong H_i(X,X\setminus\cbr{x};\ZZ)$ for all $i$ if and only if 
\[
H_i(A;\ZZ) \cong H_i(X\setminus\cbr{x};\ZZ)
\]
for all $i$.
  Consequently we have only to show that the latter assertion holds true.

  Set $\Gamma := X^{(*)} \setminus A^{(*)}$.
  The cell $\sigma$ is then the least cell in $\Gamma$.
  Note that since $X$ is regular, the closure $\overline \tau$ of each cell $\tau$ of $X$ is homeomorphic to a closed ball through the characteristic map.
  Take a maximal cell $\tau_1$ in $\Gamma$ and let $X_1$ be the closed subset of $X$ corresponding to the subcomplex $X^{(*)} \setminus \cbr{\tau_1}$.
  If $\tau_1 = \sigma$, then $X_1 = A$ and $X \setminus \cbr x = A \cup (\overline\sigma\setminus \cbr x)$. Moreover $\partial\overline\sigma = A \cap (\overline\sigma \setminus \cbr x)$ is a strong deformation retract of $\overline \sigma \setminus \cbr x$, and hence $A$ is a strong deformation retract of $X \setminus \cbr x$. In the case where $\tau_1 \neq \sigma$, it follows that $x \in \partial\overline{\tau_1}$ and
  \[
  (X_1 \setminus \cbr x) \cap (\overline{\tau_1}\setminus \cbr x) = \partial\overline{\tau_1} \setminus x.
  \]
  For a closed ball $B$ and its points $p$ in $\partial B$, $\partial B \setminus \cbr p$ is a strong deformation retract of $B \setminus \cbr p$.
  Hence $\partial\overline{\tau_1}\setminus \cbr x$ is also a strong deformation retract of $\overline{\tau_1} \setminus \cbr x$.
  Consequently $X_1 \setminus \cbr x$ is a strong deformation retract of $X \setminus \cbr x$. Thus we can replace $X$ by $X_1$ and $\Gamma$ by $X_1^{(*)} \setminus A^{(*)}$.

  This procedure stops by finite steps since $\Gamma$ is a finite set. Therefore we conclude that $H_i(A;\ZZ) \cong H_i(X\setminus\cbr{x};\ZZ)$ for all $i$, as desired.
\end{proof}

A finite regular CW complex $X^{(\ast)}$ always admits a finite simplicial complex whose geometric realization is homeomorphic to the underlying space $X$. In fact, we can take the barycentric subdivision. 
We say $X^{(\ast)}$ (or $X$)  is Cohen--Macaulay over $\kk$, if the Stanley--Reisner ring $\kk[\Delta]$ is Cohen--Macaulay.
This condition does not depend on the particular choice of $\Delta$. 
In fact, $X$ is Cohen--Macaulay over $\kk$ if and only if 
$$\widetilde{H}^i(X;\kk)= H^i(X, X \setminus \cbr{x};\kk)=0$$
for all $i < \dim X$ and all $x \in X$. This is a classical result of Munkres. See for example \cite[II. Proposition~4.3]{St}. In our setting we have the following.

\begin{lem}\label{CM criterion}
With the above notation,  $X$ is Cohen--Macaulay if and only if 
$H^i(C^\bullet(X^{\ge \sigma}))=0$ for all $i \ne \dim X$ and all $\sigma \in X^{(\ast )}$. 
\end{lem}

\begin{proof}
The assertions are immediate from Proposition~\ref{prop:link-iso}  and the above remark on the Cohen--Macaulay property of a topological space. 
\end{proof}

The following is one of the main results of the present paper.

\begin{thm}\label{CM properties}
Let $I \subset S$ be a monomial ideal. Assume that the quotient $S/I$ admits  a faithful cellular resolution $\cF_\bullet^X$. 
If $S/I$ is Cohen--Macaulay, then  the supporting complex $X$ is Cohen--Macaulay over $\kk$.  
\end{thm}

\begin{proof}
By Proposition~\ref{ring CM},  we have $H^i(C^\bullet(X^{\ge \ba}))=0$ for all $\ba \in \ZZ^n$ and all $i \ne \dim X$. 
Since $\cF_\bullet^X$ is faithful now,  we have    $X^{\ge \sigma}= X^{\ge \gr(\sigma)}$ for all $\sigma \in X^{(\ast)}$. 
Hence  it follows that $H^i(C^\bullet(X^{\ge \sigma}))=0$ for all $\sigma \in X^{(\ast)}$ and all $i \ne \dim X$. 
So the assertion follows from Lemma~\ref{CM criterion}.  
\end{proof}

\begin{cor}\label{cor: ring CM => B_M CM}
Let $O_\cM$ be the ideal associated with an affine oriented matroid $\cM$. 
If $\wS/O_\cM$ is Cohen--Macaulay, then the bounded complex $X(\cB_\cM)$ of $\cM$ is Cohen--Macaulay.  
\end{cor}

\begin{proof}
An immediate consequence of Theorems~\ref{thm:NPS-cell-res} and \ref{CM properties}. 
\end{proof}

In the next section, we will see that the Cohen--Macaulayness of $\wS/O_\cM$ does not depend on $\kk$. 

\begin{rem}\label{rem:CM-imply-CM}
  (1) Let $I \subset S$ be a strongly stable monomial ideal (i.e., a Borel fixed ideal if $\chara(\kk)=0$). For its {\it alternative polarization} $\bpol(I) \subset T$, where $T$ is a larger polynomial ring, $T/\bpol(I)$ has a minimal cellular resolution $\cF_\bullet^X$ supported by a regular CW complex $X$. If $T/\bpol(I)$ is moreover Cohen--Macaulay, then $X$ is homeomorphic to a ball and hence is Cohen--Macaulay. See \cite{OY15} for details. Because $\cF_\bullet^X$ is faithful, the Cohen--Macaulayness of $X$ also follows from Theorem~\ref{CM properties}.

(2) The converses of Theorem~\ref{CM properties} and Corollary~\ref{cor: ring CM => B_M CM} are far from true. 
For example, consider the affine oriented matroid $\cM$ given by the affine hyperplane arrangement of Figure~\ref{square}. The bounded complex $X(\cB_\cM)$ is clearly Cohen--Macaulay. On the other hand, 
by Corollary~\ref{cor:redCM} below,
$\wS/O_\cM$ is Cohen--Macaulay if and only if so is $S/\overline{O_\cM} = \kk[x_1, \ldots, x_4]/(x_1x_4, x_1x_3, x_2x_3, x_2x_4)$. The latter ring is just the Stanley--Reisner ring of two (disjoint) 1-simplexes and hence is not Cohen--Macaulay. Therefore neither is $\wS/O_\cM$.
\begin{figure}
\begin{center}
\includegraphics{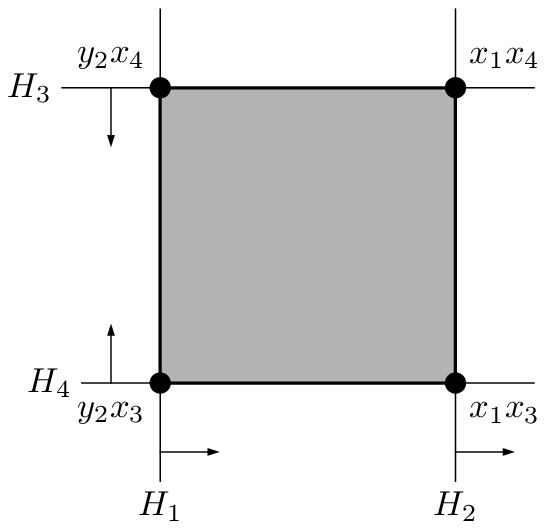}
\caption{$X(\cB_\cM)$} \label{fig:cmgap}
\label{square}
\end{center}
\end{figure}

(3) As shown in \cite{MSY}, a minimal free resolution of the quotient $S/I$ by a  {\it generic monomial ideal}  $I$ is supported by a simplicial complex $\Delta_I$ called the {\it Scarf complex} of $I$. 
It is easy to see that  $\cF_\bullet^{\Delta_I}$ is faithful.
Hence, by Theorem~\ref{CM properties}, if $I$ is  a Cohen--Macaulay generic monomial ideal, then $\Delta_I$ is Cohen--Macaulay. 
However a stronger result has been obtained. 
In fact, \cite[Theorem~2.5]{MSY} states that $\Delta_I$ is shellable in this situation. 
On the other hand, it is not homeomorphic to a ball, nor is it a homology manifold. For example, $(x^3y^2, y^3z^2, z^3x^2, xyz)$ is a Cohen--Macaulay generic monomial ideal, but $\Delta_I$ consists of 3 line segments joined at a point.   
This is in contrast to the other examples appearing in this paper (see also Corollary~\ref{cor:X-hmfd}). 

(4) We have no idea whether the  faithfulness assumption  is really necessary for Theorem~\ref{CM properties}. 
The Eliahou--Kervaire resolutions of stable monomial ideals $I$ are typical examples of 
non-faithful cellular resolutions. In this case, if $S/I$ is Cohen--Macaulay, then 
the supporting CW complex $X^{(\ast)}$ is Cohen--Macaulay. In fact, the authors have shown that $X$ is homeomorphic to a ball in this situation (\cite{OY15}).  
\end{rem}
\section{Cohen--Macaulayness of oriented matroid ideals}
From the definition of $O_\cM$, it is natural to expect the sequence $x_1-y_1,\dots ,x_n-y_n$ to be regular on $\widetilde S/O_\cM$. Novik et al.\ indeed showed that this is true when $\widetilde S/O_\cM$ is Cohen--Macaulay (\cite[Corollary 2.7]{NPS}). In this section, we will show that the sequence is \emph{always} regular on $\widetilde S/O_\cM$.
Only in this section, for an affine oriented matroid $\cM= (\br n \cup \cbr g, \cL ,g)$, we allow $g$ to be a loop, and in this case, we set $O_\cM = 0$ and $\cB_\cM = \void$.

Let us first recall the following lemma that can be found in \cite{EN}.

\begin{lem}[cf.\ {\cite[Lemma 5.8]{EN}}] \label{lem:rad-decomp}
Let $R$, $R'$ be noetherian rings and $f:R \to R'$ a homomorphism of rings. For a radical ideal $I$ of $R$ with minimal primary decomposition $I = \bigcap_{k=1}^s \fp_k$, it follows that $f(I) = \bigcap_{k=1}^sf(\fp_k)$ if $f(I)$ is radical.
\end{lem}

For $A \subseteq \br n$, we set $S[A] := S \otimes_\kk \kk\br{y_i \mid i \in A}$. Let $\cI_A$ be the family of squarefree monomial ideals $I$ of $S[A]$ satisfying the following properties:
\begin{enumerate}
\item $x_iy_i \nmid \sfm$ for all $i \in A$ and $\sfm \in G(I)$.
\item $\cbr{x_i,y_i} \not\subseteq \fp$ for all $i \in A$ and $\fp \in \Ass(S[A]/I)$.
\end{enumerate}

\begin{lem} \label{lem:reg} The sequence $\cbr{x_i - y_i}_{i \in A}$ is regular on $S[A]/I$ for all $I \in \cI_A$.
\end{lem}
\begin{proof}
Let $\varphi_i: S[A] \to S[A \setminus\cbr{i}]$ be the canonical surjective ring homomorphism sending $y_i$ to $x_i$ and the other variables to themselves.
For each $i \in A$ and ideal $J \in \cI_A$, the element $x_i - y_i$ is regular on $S[A]/J$ since $x_i - y_i \not\in \fp$ for all $\fp \in \Ass\pr{S[A]/J}$. Moreover, through the isomorphism $S[A]/(x_i-y_i) \cong S[A \setminus \cbr{i}]$ induced by $\varphi_i$, we have the isomorphism
\[
S[A]/J \otimes_{S[A]} S[A]/(x_i - y_i) \cong S[A\setminus\cbr{i}]/\varphi_i(J).
\]
Thus we have only to show that $\varphi_i(I) \in \cI_{A\setminus\cbr i}$ for all $i \in A$.

Let $i \in A$. Since $I$ is squarefree and $x_iy_i \nmid \sfm$ for all $\sfm \in G(I)$, we have the decomposition $I = \bigcap_{\fp \in \Ass\pr{S[A]/I}}\fp$ and the ideal of $\varphi_i(I)$ of $S[A \setminus \cbr{i}]$ is again squarefree. In particular $\varphi_i(I)$ is radical. It thus follows from Lemma~\ref{lem:rad-decomp} that $\varphi_i(I) = \bigcap_{\fp \in \Ass\pr{S[A]/I}}\varphi_i(\fp)$. By the definition of $\varphi_i$, each $\varphi_i(\fp)$ is also prime. Hence $\varphi_i(I)$ inherits the property of $I$ on generators and associated primes. Therefore $\varphi_i(I) \in \cI_{A\setminus\cbr{i}}$.
\end{proof}
Let $\pi: \widetilde S \to S$ be the natural surjective map given by the specialization of $y_i$ to $x_i$ for each $i \in \br n$. 
Set $\overline{O_\cM} := \pi(O_\cM)$. As one can easily verify, it follows that
\[
\overline{O_\cM} = \pr{x_{\supp(\gl) \setminus \cbr g} \mid \gl \in \cC\cap \cL^+} = \pr{x_{F \setminus \cbr g} \mid F \in \underline\cC,\ g \in F},
\]
where $x_G := \prod_{i \in G}x_i$ for $G \subseteq \br n$.

As a corollary of the proposition above, the following holds.
\begin{cor}\label{cor:redCM}
For an affine oriented matroid $\cM = (\br n \cup \cbr{g},\cL ,g)$ such that $g$ is not a coloop, the sequence $x_1-y_1,\dots ,x_n - y_n$ is always regular on $\widetilde S/O_\cM$. In particular, $\widetilde S/O_\cM$ is Cohen--Macaulay if and only if so is $S/\overline{O_{\cM}}$.
\end{cor}
\begin{proof}
We will show only the first assertion. The second is an immediate consequence of the first. By Lemma~\ref{lem:reg}, it is enough to show that $O_\cM \in \cI_{\br n}$.
Since $g$ is not a coloop, the ideal $O_{\cM}$ is a proper ideal of $\widetilde S$. Recall that $G(O_\cM) = \set{\sfm_\gl}{\gl \in \cC \cap \cL^+}$. It thus follows that $x_iy_i \nmid \sfm$ for all $\sfm \in G(O_\cM)$ and $i \in \br n$.

Suppose there exist $i \in \br n$ and $\fp \in \Ass(\widetilde S/O_\cM)$ such that $\cbr{x_i, y_i} \subseteq \fp$.
Then we can find a monomial $u \in \widetilde S \setminus O_\cM$ such that $x_iu \in O_\cM$ and $y_iu \in O_\cM$.
Hence there exist $\gl, \gm \in \cC \cap \cL^+$ such that $x_i \mid \sfm_\gl \mid x_iu$ and $y_i \mid \sfm_\gm \mid y_iu$.
By \cite[Theorem 3.2.5]{BVSWZ}, we have $\nu \in \cC \cap \cL^+$ such that $\nu(i) = 0$ and $\nu(k) = \gl(k)$ or $\nu(k) = \gm(k)$ for $k \in \supp(\nu)$. However this implies that $\sfm_\nu \mid (\lcm(\sfm_\gl,\sfm_\gm)/(x_iy_i))$ and hence that $\sfm_\nu \mid u$. This is a contradiction.
\end{proof}
For  $i \in \br n$, set  $\cL' := \{ \gl \in \cL \mid \gl(i)=0 \}$, and we sometimes regard it as a subset of $\signs^{E \setminus \cbr{i}}$. 
Then  $\cM':=(E \setminus \{ i \}, \cL', g)$  
is an affine oriented matroid again.  This is  the {\it contraction} $\cM/\cbr{i}$ of $\cM$ to $E \setminus \cbr{i}$ (see \cite[Lemma 4.1.8]{BVSWZ}).
Clearly, the bounded complex $\cB_{\cM'}$ of $\cM'$ can be identified with 
$$\cB_\cM \cap \cL' = \{ \gl \in \cB_\cM \mid \gl(i) = 0 \}.$$

\begin{lem}\label{rank lemma}
If   $\cB_\cM  \not \subset \cL'$ (that is, there is some $\gl \in \cB_\cM$ with $\gl(i) \ne 0$), then we have $\rank \cB_{\cM'} < \rank \cB_\cM$. 
\end{lem}

\begin{proof}
  By \cite[Theorem 3.2]{D}, any maximal element $\mu \in \cB_\cM$ satisfies $\mu(i) \ne 0$. 
Hence $\mu \not \in  \cB_{\cM'}$, and the assertion follows. 
\end{proof}

Set $\wS':=\kk[x_j, y_j \mid  j \in \br{n} \setminus \cbr{i}] \subset \wS$.  Then $O_{\cM'}$ is an ideal of $\wS'$. 

\begin{lem}\label{codim 1}
In the above situation, if $\wS/O_\cM$ is Cohen--Macaulay, then $\wS'/O_{\cM'}$ is also. 
\end{lem}

\begin{proof}
The case where $\cB_{\cM'} = \void$ is clear. Assume $\cB_{\cM'} \neq \void$. Clearly, $g$ is not a coloop also in $\cM'$. So the results in the previous sections hold for $\wS'/O_{\cM'}$. 

If $\cB_\cM \subset \cL'$, then we have $O_\cM= O_{\cM'}\widetilde{S}$ and $\widetilde{S}/O_\cM \cong (\widetilde{S}'/O_{\cM'})[x_i, y_i]$. 
So the assertion is clear. Now we assume that $\cB_\cM \not \subset \cL'$. 
Then we have  $\rank \cB_{\cM'} < \rank \cB_\cM$ by Lemma~\ref{rank lemma}. 

Recall that $\wS$ (resp. $\wS'$) is a $\ZZ^{2n}$-graded  (resp. $\ZZ^{2n-2}$-graded) ring. 
The inclusion  $\wS' \subset \wS$ induces the injection $\ZZ^{2n-2} \hookrightarrow \ZZ^{2n}$, and we regard 
$\ZZ^{2n-2}$ as a subset of $\ZZ^{2n}$ in this way.  
For simplicity, we set $X^{(\ast)} := X(\cB_\cM)$ and $Y^{(\ast)} := X(\cB_{\cM'})$. Then we can regard $Y^{(\ast)}$ as a subcomplex of $X^{(\ast)}$.  
Recall that we defined the order filter $X^{\ge \ba} := \{ \sigma \in X^{(\ast)} \mid \gr(\sigma) \ge \ba \}$ of $X^{(\ast)}$ from $\ba \in \ZZ^{2n}$.  

For $\bb \in \ZZ^{2n-2}$, the order filter $Y^{\ge \bb}$ of $Y^{(\ast)}$ is defined in a similar way. 
Then we can regard $Y^{\ge \bb} \subseteq X^{\ge \bb}$ through the inclusion $\ZZ^{2n-2} \hookrightarrow \ZZ^{2n}$.  Set $\bb+i:= \bb + \deg(x_i) \in \ZZ^{2n}$ and  $\bb-i:= \bb + \deg(y_i) \in \ZZ^{2n}$. 
Then we have 
$$X^{\ge \bb}=X^{\ge \bb+i} \sqcup X^{\ge \bb-i} \sqcup Y^{\ge \bb},$$ 
and it yields an exact sequence 
\begin{equation}\label{B_M>gl 's}
0 \longto C^\bullet(X^{\ge \bb+i}) \oplus C^\bullet(X^{\ge \bb- i}) \longto  C^\bullet(X^{\ge \bb}) \longto  C^\bullet(Y^{\ge \bb}) \longto 0
\end{equation}
of cochain complexes. 

Since $\widetilde{S}/O_\cM$ is Cohen--Macaulay, we have 
$$H^j( C^\bullet(X^{\ge \bb+i}))=H^j(C^\bullet(X^{\ge \bb-i})) = H^j( C^\bullet(X^{\ge \bb})) =0$$
for  all $j \ne \dim X = \rank \cB_\cM$ by Proposition~\ref{ring CM}. 
From the sequence \eqref{B_M>gl 's}, we have  $H^j( C^\bullet(Y^{\ge \bb})) =0$ for all $j < \rank \cB_{\cM} - 1$. 
By the present assumption $\rank \cB_{\cM'} < \rank \cB_\cM$,  it means that  $H^j( C^\bullet(Y^{\ge \bb})) =0$ for all $j < \rank \cB_{\cM'}$. 
Since clearly $H^j(C^\bullet(Y^{\ge \bb})) = 0$ for all $j > \rank \cB_{\cM'}$, the assertion follows from Proposition~\ref{ring CM}. 
\end{proof}

For $A \subseteq \br n$, we set $S_A := \kk[x_i \mid i \in A]$ and $\widetilde S_A := S_A \otimes_\kk \kk[y_i \mid i \in A]$. Recall that an ordinary matroid $M$ on a finite set $V$ is characterized by the independent sets, which form a simplicial complex $\gD$ over $V$ whose facets are just the bases of $M$ (See \cite{St} for details). In the sequel, we refer $(V,\gD)$ or simply $\gD$ to a matroid. 

\begin{thm} \label{thm:CM-implies-matroid} 
If $\widetilde{S}/O_\cM$ is Cohen--Macaulay, then $S/\overline{O_\cM}$ is the Stanley--Reisner ring of a matroid. 
\end{thm}

\begin{proof}
Since $\overline{O_\cM}$ is a squarefree monomial ideal of $S$, there is a simplicial complex $\Delta \subset 2^{[n]}$ whose Stanley--Reisner ideal $I_\Delta$ coincides with  $\overline{O_\cM}$. 
Recall that $\Delta$ is a matroid if and only if $\Delta|_F$ is Cohen--Macaulay for all $F \subset [n]$ (\cite[Proposition 3.1 in Chap. III]{St}).

Set  $$\cL(F \cup \{g \}):= \{ \gl |_{F \cup \{g \}} \mid \gl \in \cL, \, \supp(\gl) \subseteq F \cup \{g \} \}.$$ 
This is the set of covectors of the contraction  $\cM''$ of $\cM$ to $F \cup \cbr{g}$ ($\cM''$ is also an affine oriented matroid).    
It is easy to see that the Stanley--Reisner ring  $S_F/I_{\Delta|_F}$ of $\Delta|_F$ coincides with 
$$\wS_F/O_{\cM''} \otimes_{\wS_F}\wS_F/(x_i -y_i \mid i \in F).$$
Since $\{ x_i -y_i \mid i \in F \}$ forms an $\wS_F/O_{\cM''}$-regular sequence by Corollary~\ref{cor:redCM}, it suffices to show that $\wS_F/O_{\cM''}$ is Cohen--Macaulay. 
We can prove this by the induction on $n - |F|$ using Lemma~\ref{codim 1}, which corresponds to the case $n-|F|=1$.  
\end{proof}

By Theorem~\ref{thm:CM-implies-matroid}, we see that the Cohen--Macaulayness of $\wS/O_\cM$ does not depend on the base field $\kk$. 

As Example~\ref{exmp:nongp-cm-arr} shows, the converse of Theorem~\ref{thm:gen-pos} does not hold in general. In the rest, we will prove that the converse is ``essentially'' true and also give a combinatorial characterization of Cohen--Macaulayness of $\widetilde S/O_\cM$.

Before that, let us recall some properties of ordinary matroids for later use.

\begin{lem}[cf.\ {\cite[Cor. 1.2.6 and Exer. 4 in Sect. 2 of Chap. 1]{O}}] \label{lem:base-circ}
Let $M := (\br n,\gD)$ be an ordinary matroid with the base set $\cF(\gD)$ and the circuit set $\cC(M)$.
\begin{enumerate}
\item For any base $F \in \cF(\gD)$ and $i \in \br n \setminus F$, there exists a unique circuit $C(i,F) \in \cC(M)$ such that $C(i,F) \subseteq F \cup \cbr{i}$.
\item For any circuit $C \in \cC(M)$ and $i \in C$, there exists a base $F \in \cF(\gD)$ such that $C = C(i,F)$.
\end{enumerate}
\end{lem}
Though the following property of circuits of ordinary matroids is probably well-known for specialists in matroid theory, we will give a proof for completeness.

\begin{cor}\label{cor:eq-circ}
Let $M = (\br n, \gD)$, $M' := (\br n, \gD')$ be ordinary matroids and $\cC(M)$, $\cC(M')$ the sets of circuits of $M$, $M'$, respectively. If $\rank M = \rank M'$ and $\cC(M') \subseteq \cC(M)$, then $M = M'$.
\end{cor}
\begin{proof}
It suffices to show that $\cC(M) \subseteq \cC(M')$. Take any $C \in \cC(M)$ and $i \in C$. By Lemma~\ref{lem:base-circ}, there exists a base $F \in \cF(\gD)$ such that $C(i,F) = C$. It follows from the definition of $C(i,F)$ that $i \not\in F$ and $C(i,F) \subseteq F \cup \cbr{i}$. Since $\# F < \#(F \cup \cbr i)$ and $\rank M = \rank M'$, the set $F \cup \cbr i$ does not belong to $\gD'$. Hence there exists a circuit $C' \in \cC(M')$ such that $C' \subseteq F \cup \cbr{i}$. Now $C' \in \cC(M)$ since $\cC(M') \subseteq \cC(M)$. Consequently, it follows from the uniqueness of $C = C(i,F)$ that $C = C'$, and hence $C \in \cC(M')$.
\end{proof}

Let $\cM := (\br n \cup \cbr g, \cL, g)$ be an affine oriented matroid with the cocircuit set $\cC$. For simplicity, we say $\cM$ is of \emph{full rank} if for any $e \in \br n \cup \cbr g$ that is not a loop, there exists $\gl \in \cB_\cM$ with $e \in \supp(\gl)$. This is equivalent to say that the rank of $\cB_\cM$ is equal to $\rank \cM - 1$.

Recall that
\[
\cC|_{\br n} = \Min\set{\gl|_{\br n}}{\gl \in \cC},
\]
where $\Min$ denotes the set of inclusion-minimal elements.
In the sequel, we set
\[
\underline{(\cC \cap \cL^+)|_{\br n}} := \set{\supp(\gl) \cap \br n}{\gl \in \cC \cap \cL^+}. 
\]

\begin{thm} \label{thm:char-genpos}
Let $\cM := (\br n \cup \cbr g, \cL, g)$ be an affine oriented matroid of rank $r$ with the cocircuit set $\cC$, and let $\cB_\cM$ be the bounded complex of $\cM$. Assume $g$ is not a coloop and $\cM$ is of full rank. Then the following conditions are equivalent:
\begin{enumerate}
\item $g$ is in general position.
\item $\widetilde S/O_\cM$ is Cohen--Macaulay.
\item $S/\overline{O_{\cM}}$ is Cohen--Macaulay.
\item $\underline{(\cC \cap \cL^+)|_{\br n}}$ is the set of circuits of some ordinary matroid of rank $n - r$.
\item $\underline{(\cC \cap \cL^+)|_{\br n}} = \underline{\cC|_{\br n}}$.
\end{enumerate}
If these are the cases, $\dim \widetilde S/O_\cM = 2n - r$ and $\dim S/\overline{O_\cM} = n - r$.
\end{thm}
\begin{proof}
The implication (1) $\imply$ (2) is just a part of Theorem~\ref{thm:gen-pos}. The equivalence (2) $\Leftrightarrow$ (3) is the same as Corollary~\ref{cor:redCM}. The implication (2) $\imply$ (4) follows from Theorem~\ref{thm:CM-implies-matroid}.

Let us prove (4) $\imply$ (5). Since $\cM$ is of rank $r$ and $g$ is not a coloop, the rank of $\cM|_{\br n}$ is equal to $r$, which implies that the underlying matroid $\underline{\cM|_{\br n}}$ of $\cM|_{\br n}$ (see Section~\ref{sec:prelim}) is of rank $r$.
The dual matroid of $\underline{\cM|_{\br n}}$ is thus of rank $n-r$ and has the circuit set $\underline{\cC|_{\br n}}$. Now $\underline{(\cC \cap \cL^+)|_{\br n}} \subseteq \underline{\cC|_{\br n}}$ since $(\cC \cap \cL^+)|_{\br n} \subseteq \cC|_{\br n}$. Therefore the desired equality follows from Corollary~\ref{cor:eq-circ}.

We will prove (5) $\imply$ (1). Assume the assertion (5) holds.
Take any $\gm \in \cC$ with $\gm|_{\br n} \in \cC|_{\br n}$. It suffices to prove $g \in \supp(\gm)$.
By the assumption, there exists $\gl \in \cC \cap \cL^+$ such that $\supp(\gl|_{\br n}) = \supp(\gm|_{\br n})$. It then follows that $\supp(\gm) \subseteq \supp(\gl)$, and hence $\gm = \gl$ or $\gm = -\gl$ by Circuit Axioms. Therefore $g \in \supp(\gm)$.
\end{proof}

\begin{rem} \label{rem:red-to-full-rank-case}
(1) Theorem~\ref{thm:char-genpos} is generalized for an affine oriented matroid $\cM := (\br n \cup \cbr g, \cL, g)$ that is not necessarily of full rank. Set
\[
F := \set{e \in \br n \cup \cbr g}{\gl(e) \neq 0 \text{ for some } \gl \in \cB_\cM}
\]
and $F' := F \setminus \cbr{g}$. It then follows that
\begin{equation}
\widetilde S/O_\cM \cong \widetilde S_{F'}/O_{\cM|_{F}} \otimes_\kk \kk[x_i,y_i \mid i \in \br n \setminus F'] \label{eq:poly-ext}
\end{equation}
Hence the theorem still holds true after a suitable modification.

(2) For positive integers $n,l$ with $n \ge l$, let $I_{n,l} \subset S=\kk[x_1, \ldots, x_n]$ be the ideal 
$$(x_F \mid F \subset \br{n}, \#F =l).$$ 
$I_{n,l}$ is given by the specialization of the ideal 
$O_\cM \subset \wS$ associated with a uniform oriented matroid $\cM$. 
Hence a minimal $S$-free resolution of $S/I_{n,l}$ is supported by the bounded complex $\cB_\cM$ of $\cM$, 
whose underlying space is homeomorphic to a ball as shown by Dong \cite{D08}. 
For example, a minimal free resolution of $I_{4,2}$ is supported by the shaded part of Figure~\ref{4lines}. 
Since $I_{n,l}$ is a Cohen--Macaulay squarefree strongly stable monomial ideal, it also follows from  \cite[Corollary~6.1]{OY15} that a free resolution is supported by a ball. We have no idea on the relation between these two constructions. 
\end{rem}

\section{Canonical module of $\wS/O_\cM$ and the topology of $X(\cB_\cM)$}
Let $\cM= (E, \cL, g)$ with $E=\br{n} \cup \cbr{g}$ be an affine oriented matroid as above. For a simple exposition, in the rest of the paper, we assume that $\cB_\cM$ is of full rank. The general case follows from the same argument as Remark~\ref{rem:red-to-full-rank-case}.
We leave the details to the reader as easy exercises.

Since $\cB_\cM$ is of full rank and is moreover pure as is shown in \cite{D}, a maximal element of $\cB_{\cM}$ is also maximal in $\cL$. 
According to Chapter 4 of \cite{BVSWZ}, we call a maximal element of $\cL$ (resp.  $\cB_\cM$) a {\it tope} of  $\cL$ (resp. $\cB_\cM$). 
If $\gl \in \cL$ is covered by a tope, we call $\gl$ a {\it subtope}.   
Any subtope of $\cL$ is contained in exactly two topes of $\cL$. 

Recall that $\cB_\cM$ gives a regular CW complex $X^{(\ast)}:= X(\cB_\cM)$. We can define the boundary $\partial X^{(\ast)}$ 
of $X^{(\ast)}$ in the natural way, and we identify $\partial X^{(\ast)}$ with the corresponding subset of $\cB_\cM$ by abuse of notation. Here $\gl \in \cB_\cM$ (more precisely, the corresponding cell in $X^{(\ast)}$) belongs to  $\partial X^{(\ast)}$  if and only if there is some $\gm \in  \cL \setminus \cB_\cM$ with $\gm > \gl$.  Clearly, $\partial X^{(\ast)}$ is also a regular CW complex, and its underlying space is a closed subset of that of $X^{(\ast)}$. 
Any subtope of $\cB_\cM$ is contained in at most two topes of $\cB_\cM$. A subtope of $\cB_\cM$ belongs to the boundary if and only if it is contained in a sole tope of $\cB_\cM$. 

Next, under the assumption that $\wS/O_\cM$ is Cohen--Macaulay, 
we will give an explicit description of the canonical module $\omega_{\wS/O_\cM}$ of  $\wS/O_\cM$. 
For $\gl \in \cL$, set $$\sn_\gl:=\frac{\prod_{i=1}^nx_iy_i}{\sfm_\gl} \in \wS.$$

\begin{thm}\label{canonical}
If  $\wS/O_\cM$ is Cohen--Macaulay, its canonical module  $\omega_{\wS/O_\cM}$ is isomorphic to the ideal 
$$J_\cM:=( \sn_\gl \mid \text{$\gl$ is a tope of $\cB_\cM$})$$
of  $\wS/O_\cM$. 
\end{thm}

To prove the theorem, we need the following lemmas. 

\begin{lem}\label{basic of n-monomial} 
Assume that $\wS/O_\cM$ is Cohen--Macaulay. Then we have the following. 
\begin{enumerate}
\item   $\sn_\mu \in O_\cM$ for any $\mu \in \cL^+ \setminus \cB_\cM$.
\item If $\gl \in \cB_\cM$ corresponds to a cell in the boundary $\partial X(\cB_\cM)$, we have $\sn_\gl \in O_\cM$. 
\end{enumerate}
\end{lem}
\begin{proof}
(1) Take any $\mu \in \cL^+ \setminus \cB_\cM$. It suffices to show that there exists $\xi \in \cC$ such that $\xi(g) = -$ and $\mu|_{\br n} \ge \xi|_{\br n}$, because it then follows that $-\xi \in \cC \cap \cL^+$ and $\sfm_{-\xi} \mid \sn_{\mu}$. By the definition of $\cB_\cM$, we have a cocircuit $\nu \in \cC$ with $\nu(g) = 0$ and $\mu > \nu$. By Theorem~\ref{thm:char-genpos}, $g$ is in general position, since $\widetilde S/O_{\cM}$ is Cohen--Macaulay 
(we assume that $\cM$ is of full rank in this section).
Hence there exists $\nu_1 \in \cC$ such that $g \in \supp(\nu_1)$ and $\nu|_{\br n} \ge \nu_1|_{\br n}$. If $\nu_1(g) = -$, then $\nu_1$ satisfies the desired condition.

Assume $\nu_1(g) = +$.
Set $\tilde\nu := \nu \circ (-\nu_1)$. It then follows from $\nu(g) = 0$ and $\supp(\nu|_{\br n}) \supseteq \supp(\nu_1|_{\br n})$ that $\tilde\nu|_{\br n} = \nu|_{\br n}$ and $\tilde\nu(g) = -$. By \cite[Proposition 3.7.2]{BVSWZ}, there exists $\rho_1,\dots, \rho_k \in \cC$ such that $\tilde\nu = \rho_1 \circ \rho_2 \circ \cdots \circ \rho_k$ and $\rho_i(e)\rho_j(e) \ge 0$ for all $i,j$ and all $e \in \supp(\tilde\nu)$. Clearly, $\nu \ge \rho_i$ for all $i$ and there exists $i_0$ such that $\rho_{i_0}(g) = -$. Since $\mu|_{\br n} \ge \nu|_{\br n} = \tilde\nu|_{\br n} \ge \rho_{i_0}|_{\br n}$, the sign vector $\rho_{i_0}$ satisfies the desired condition.

(2)  Since $\gl$ is in the boundary of $\cB_\cM$, there exists $\mu \in \cL^+ \setminus \cB_\cM$ with $\mu > \gl$, or equivalently $\sn_\mu \mid \sn_\gl$. So the assertion directly follows from (1).  
\end{proof}

Recall that $X^{(\ast)} =X(\cB_\cM)$ supports a minimal $\ZZ^{2n}$-graded $\wS$-free resolution $\cF_\bullet^X$ of $\wS/O_\cM$. 
Let  $\omega_{\wS} := \wS(-{\mathbf 1})$ with ${\mathbf 1}=(1, \ldots, 1) \in \ZZ^{2n}$ be the $\ZZ^{2n}$-graded canonical module of $\wS$.
If $\wS/O_\cM$ is Cohen--Macaulay, then 
$$
H^i(\Hom_{\wS}^\bullet(\cF^X_\bullet, \omega_{\wS})) \cong 
\begin{cases}
\omega_{\wS/O_\cM} & \text{if $i= \dim X+1$,}\\
0 & \text{otherwise.}
\end{cases}
$$
Hence  $\Hom_{\wS}^\bullet(\cF^X_\bullet, \omega_{\wS})$ gives a minimal $\ZZ^{2n}$-graded  $\wS$-free resolution of $\omega_{\wS/O_\cM}$ up to shift. We also remark that $$\Hom_{\wS}(\wS(-\deg(\sfm_\gl)), \omega_{\wS}) \cong \wS(-\deg(\sn_\gl))$$ 
for each $\gl \in \cL$. As in the proof of Theorem~\ref{CM properties}, 
if $\gl  \in \cB_\cM$ corresponds to $\sigma \in X^{(\ast)}$, we have  
\begin{equation}\label{component of omega}
H^d(C^\bullet(X^{\ge \sigma})) \cong [\omega_{\wS/O_\cM}]_{\deg(\sn_\gl)},
\end{equation}
where $d:= \dim X$, and $C^\bullet(X^{\ge \sigma})$ is the cochain complex defined in \S 3. 
The minimal presentation of $ \omega_{\wS/O_\cM}$ is of the form 
\begin{equation}\label{presentation of w}
\bigoplus_{\gm: \text{ subtope of $\cB_\cM$}} \wS(-\deg(\sn_\gm)) \stackrel {\psi}{\longto} 
 \bigoplus_{\gl: \text{ tope of $\cB_\cM$}} \wS(-\deg(\sn_\gl)) \longto  \omega_{\wS/O_\cM} \longto 0. 
\end{equation} 
Here $\psi$ sends the basis element $e_\gm$ of a free summand $\wS(-\deg(\sn_\gm))$ to 
$$\sum_{\gl: \text{ tope of $\cB_\cM$}}  [\lambda: \mu] \cdot (\sn_\gm/\sn_\gl) \, e_\gl \in  \bigoplus_{\gl: \text{ tope of $\cB_\cM$}}  \wS(-\deg(\sn_\gl)).$$
Note that, in this situation,  $ [\lambda: \mu] \ne 0$, if and only if $\gm < \lambda$, if and only if $\sn_\gl$ divides $\sn_\gm$. 

In the following result, the assumption that $\wS/O_\cM$ is Cohen--Macaulay is superfluous, while we do not omit it here for simple exposition.



\begin{lem}\label{generator}
Assume that $\wS/O_\cM$ is Cohen--Macaulay. 
Let  $\Delta_\cM \subset 2^{\br {2n}}$ be the simplicial complex whose Stanley--Reisner ring  is $\wS/O_\cM$.
For a squarefree monomial $\sfm \in \wS/O_\cM$ whose support corresponds to a facet of $\Delta_\cM$, we have $\sfm \in J_\cM$. 
Here $J_\cM$ is the ideal of $\wS/O_\cM$ defined in Theorem~\ref{canonical}.  
\end{lem}

\begin{proof}
Set $d':= \dim \wS/O_\cM =\dim \Delta_\cM +1$, and let $\mathfrak m := (x_i,y_i \mid i \in \br n)$ be the graded maximal ideal of $\wS$. 
For $\ba \in \NN^{2n}$,  we have $\dim_\kk [\omega_{\wS/O_\cM}]_{\ba}= \dim_\kk [H_{\mathfrak m}^{d'}(\wS/O_\cM)]_{-\ba}$ by the local duality theorem, and the right side can be computed by Hochster's formula (c.f. \cite[Theorem~5.3.8]{BH}). Here $H_{\mathfrak m}^{d'}(\wS/O_\cM)$ denotes the $d'$-th local cohomology module with respect to $\mathfrak m$. If $\ba = \deg(\sfm)$ (that is, the support of 
$\ba$ is a facet of $\Delta_\cM$), then $\dim_\kk [H_{\mathfrak m}^{d'}(\wS/O_\cM)]_{-\ba}=1$ and hence 
$[\omega_{\wS/O_\cM}]_{\deg(\sfm)} \ne 0$. Since any minimal homogeneous generator of $\omega_{\wS/O_\cM}$ has the 
same degree as $\sn_\gl$ for some tope  $\gl \in \cB_\cM$ by \eqref{presentation of w}, 
$\sfm$ can be divided by $\sn_\gl$ for some tope $\gl \in \cB_\cM$. Since $\sn_\gl \in J_\cM$, we are done. 
\end{proof}

\begin{proof}[Proof of Theorem~\ref{canonical}.] 
Recall that $\cL$ gives a regular CW complex  whose underlying space is homeomorphic to a $d$-sphere. Hence $\widehat{\cL} := \cL \cup \{ \hat{1}\}$ gives a regular CW complex $X(\widehat{\cL})$ whose underlying space is homeomorphic to a $(d+1)$-ball. By abuse of notation, $\hat{1}$ also denotes the maximum cell of  $X(\widehat{\cL})$.
Clearly, $X^{(\ast)}:= X(\cB_\cM)$ is a subcomplex of   $X(\widehat{\cL})$. 
In the sequel, as an incidence function of $X^{(\ast)}$, we use the restriction of an incidence function of $X(\widehat{\cL})$.

Define a  $\ZZ^{2n}$-graded $\wS$-morphism 
$$\phi: \bigoplus_{\gl: \text{ tope of $\cB_\cM$}} \wS(-\deg(\sn_\gl)) \longto J_\cM$$
by $e_\gl \longmapsto [\hat{1} : \gl] \cdot \sn_\gl  \in J_\cM$, where $e_\gl$ is a free basis corresponding to a direct summand $\wS(-\deg(\sn_\gl))$. 
Clearly, $\phi$ is surjective. For the map $\psi$ appearing in the minimal presentation \eqref{presentation of w} of $\omega_{\wS/O_\cM}$, we will show that $\phi \circ \psi =0$.  To do this, it suffices to show that $\phi \circ \psi(e_\gm)=0$ for all subtope $\gm$ of $\cB_\cM$.

Let $\gl \in \cB_\cM$ be a tope, and $\gm \in \cB_\cM$ a subtope with $\gm  < \gl$. Clearly, $\sn_\gl$ divides $\sn_\gm$. 
If $\gm$ corresponds to a cell in  $\partial X^{(\ast)}$, then $(\sn_\mu/\sn_\gl)\cdot \sn_\gl =\sn_\gm \in O_\cM$ by Lemma~\ref{basic of n-monomial} (2).  Hence $[J_\cM]_{\deg(\sn_\gm)}=0$,  and we have $\phi \circ \psi(e_\gm)=0$. 
Next we consider the case $\gm$ does  not belong to  $\partial X^{(\ast)}$.
In this case, $\gm$ is contained in exactly two 
topes of $\cB_\cM$, say $\gl_1$ and $\gl_2$.
By a property of the incidence function, $\phi \circ \psi$ sends $e_\gm \in \wS(-\deg(\sn_\mu))$ to 
\[
  ([\hat{1} : \lambda_1] \cdot [\lambda_1 : \mu ] + [\hat{1} : \lambda_2] \cdot [\lambda_2 : \mu ]) \cdot \sn_\gm = 0.
  \]
Now we have $\phi \circ \psi =0$. Hence we have a $\ZZ^{2n}$-graded $\wS$-homomorphism 
$$f:   \omega_{\wS/O_\cM} (\cong \opn{coker} (\psi)) \longto J_\cM$$
by the universal property of the cokernel. 
Since $\phi$ is surjective, $f$ is also. So it suffices to show that it is injective. 

For the contradiction, assume that $\ker(f) \ne 0$. 
Then
\[
\dim (\ker(f)) = \dim \omega_{\wS/O_\cM} = \dim \Delta_\cM +1,
\]
where $\Delta_\cM$ is the simplicial complex defined in Lemma~\ref{generator}.   
Since $\ker(f)$ is a {\it squarefree module} over $\wS$ (see \cite{Y00} for the definition and basic properties), there is  a facet $F \in \Delta_\cM$ with $[\ker(f)]_F \ne0$. Here we identify $F \subset \br{2n}$ with the 
corresponding squarefree vector in $\ZZ^{2n}$.  On the other hand, we have  $[J_\cM]_F \ne 0$ for all facets $F \in \Delta_\cM$ 
by Lemma~\ref{generator}. 
Similarly, we have shown that $\dim_\kk [\omega_{\wS/O_\cM}]_F =1$ for all facets $F \in \Delta_\cM$ in the proof of Lemma~\ref{generator}. 
Hence we have $[\ker(f)]_F =0$ for all facets $F$, and this is a contradiction. 
\end{proof}
%


\section{$\wS/O_\cM$ is Cohen--Macaulay implies $\cB_\cM$ is a homology manifold}

In this section, under the assumption that $\widetilde S/O_\cM$ is Cohen--Macaulay, we will prove that $X(\cB_\cM)$ and its boundary $\partial X(\cB_\cM)$ are homology manifolds over $\ZZ$. 
Moreover,  the same is true for (the geometric realization of) the simplicial complex $\Delta_\cM$ whose Stanley--Reisner ring coincides with $\wS/O_\cM$, while the dimensions of $X(\cB_\cM)$ and $\Delta_\cM$ are different.    
To obtain these results, the description of $\omega_{\wS/O_\cM}$ given in the previous section plays a role. 

Let us first recall the definition of homology manifolds. There are several slightly different definitions (cf.\ \cite{Mi,Mu,Sp}).
Here we eclectically adopt the definition from \cite{Mi} and \cite{Mu}.
There is no problem in doing so since in this paper we deal with only finite regular CW complexes.

A finite regular CW complex $X$ is said to be a \emph{homology $n$-manifold} over an abelian group $G$ if
\[
H_i(X,X\setminus\cbr{x};G) \cong
\begin{cases}
  0 & \text{for $i \neq n$,} \\
  0 \text{ or }G & \text{for $i = n$.}
\end{cases}
\]
If this is the case, the subset
\begin{equation}
  \set{x \in X}{H_n(X,X\setminus\cbr{x};G) \cong 0} \label{eq:bd}
\end{equation}
is called the boundary of $X$.

The next lemma is probably well-known for specialists, but let us introduce it for completeness. The proof is an easy exercise using the universal coefficient theorem and the structure theorem of finitely generated abelian groups.

\begin{lem}\label{lem:redhom-lem}
Let $(X,A)$ be a topological pair and $n$ an integer. Assume $H_i(X,A;\ZZ)$ is finitely generated for all $i$. Then the following are equivalent.
\begin{enumerate}
\item $H_i(X,A;\ZZ) = 0$ for all $i < n$.
\item $H_i(X,A;\kk) = 0$ for all $i < n$ and any field $\kk$.
\end{enumerate}
If these are the cases, then $H_n(X,A;\ZZ) \cong \ZZ$ if and only if $\dim_\kk H_n(X,A;\kk) = 1$ for any field $\kk$.
\end{lem}

Let $\cM$ be an affine oriented matroid with the bounded complex $\cB_\cM$. The following lemma in conjunction with Proposition~\ref{prop:link-iso} implies that $X(\cB_\cM)$ is a homology manifold over $\kk$ when $\wS/O_\cM$ is Cohen--Macaulay.

\begin{lem}\label{interior}
Assume that $\wS/O_\cM$ is Cohen--Macaulay.  For a cell $\sigma \in X^{(\ast)} := X(\cB_\cM)$,  we have  
$$
H^i(C^\bullet(X^{\ge \sigma})) \cong  
\begin{cases}
\kk & \text{if $i= \dim X$ and $\sigma \not \in \partial X^{(\ast)}$,} \\
0 & \text{otherwise.}
\end{cases}
$$
\end{lem}

\begin{proof}
Since $X^{(\ast)}$ is Cohen--Macaulay by Theorem~\ref{CM properties}, $H^i(C^\bullet(X^{\ge \sigma})) = 0$ for all $i  < d: =\dim X$ by Lemma~\ref{CM criterion}. 
So it suffices to consider $H^d(C^\bullet(X^{\ge \sigma}))$. First, assume that $\sigma \not \in \partial X^{(\ast)}$.  
Recall that the covector set $\cL  \, (\supset \cB_\cM)$ is also a CW poset, and gives a regular CW complex $Y^{(\ast)}$ whose underlying space is homeomorphic to a sphere (Theorem~\ref{thm:TRT}). 
Note that $X^{(\ast)}$ is a subcomplex of $Y^{(\ast)}$ with $\dim Y^{(\ast)}= \dim X^{(\ast)} (=d)$. 
Since $\sigma \in X^{(\ast)} \setminus \partial X^{(\ast)}$, we have $X^{\ge \sigma} = Y^{\ge \sigma}$. Since $Y^{(\ast)}$ is a $d$-sphere, we have 
$$H^d(C^\bullet(X^{\ge \sigma})) \cong H^d(C^\bullet(Y^{\ge \sigma})) \cong \kk.$$
The second isomorphism follows from Proposition~\ref{prop:link-iso}.   

Next, assume that $\sigma \in \partial X^{(\ast)}$. Let $\gl \in \cB_\cM$ be the covector corresponding to $\sigma$. 
Then $[\wS/O_\cM]_{\deg(\sn_\gl)} = 0$ by Lemma~\ref{basic of n-monomial} (2), so we have $[\omega_{\wS/O_\cM}]_{\deg(\sn_\gl)} = 0$. 
Hence the assertion follows from \eqref{component of omega}. 
\end{proof}

As an immediate consequence, the following result can be deduced.

\begin{cor}\label{cor:X-hmfd}
Assume $\widetilde S/O_\cM$ is Cohen--Macaulay (for some field $\kk$), then $X := X(\cB_\cM)$ is a homology manifold with boundary over $\ZZ$. Moreover the boundary defined by \eqref{eq:bd} coincides with the natural one.
\end{cor}
\begin{proof}
  As is stated below Theorem~\ref{thm:CM-implies-matroid}, the Cohen--Macaulayness of $\widetilde S/O_\cM$ does not depend on the base field $\kk$.
  Hence it follows from Proposition~\ref{prop:link-iso} and Lemmas~\ref{lem:redhom-lem}, \ref{interior} that $X$ is a homology $(\dim X)$-manifold with boundary
\[
\bigcup_{\sigma \in \partial X^{(*)}}\sigma = \set{x \in X}{H_{\dim X}(X,X \setminus \cbr x;\kk) \cong 0}
\]
over $\ZZ$.
\end{proof}

For $\gl \in \cB_\cM$, we set
\[
\cB_\cM^{>\gl} := \set{\mu \in \cB_\cM}{\mu > \gl}.
\]
In the case where $\rank \cB_\cM = 2$, the Cohen--Macaulayness of $X(\cB_\cM)$, hence that of $\widetilde S/O_\cM$ (Corollary~\ref{cor: ring CM => B_M CM}) implies $X(\cB_\cM)$ to be homeomorphic to a closed ball (cf. Figures~\ref{4lines}, \ref{fig:pp-cm-arr}, \ref{fig:pp-noncm-arr}, and \ref{fig:cmgap}).

\begin{prop} \label{prop:rank2}
Assume $\rank \cB_\cM = 2$. Then $X(\cB_\cM)$ is homeomorphic to a closed ball if and only if it is Cohen--Macaulay over a field $\kk$ as a topological space.
\end{prop}
\begin{proof}
  We have only to show ``if'' part. By the arguments in \cite[Sections 3 and 4]{D08}, it suffices to show that $\gD(\cB_\cM^{>\gl})$ is empty, or otherwise either a PL sphere or a PL ball, for any $\gl \in \cB_\cM$.

Assume $\gD(\cB_\cM^{>\gl})$ is not empty. Since $\gD(\cB_\cM^{>\gl})$ is the link of some face of $\gD(\cB_\cM)$, it is also Cohen--Macaulay and moreover its dimension is less than or equal to one. Therefore $\gD(\cB_\cM^{>\gl})$ is shellable.
Furthermore it follows from \cite[Theorem 4.1.14]{BVSWZ} that each face of codimension one is contained at most two facets. These two facts indeed imply that $\gD(\cB_\cM^{>\gl})$ is either a PL sphere or a PL ball (see \cite[Proposition 4.7.22]{BVSWZ} for example).
\end{proof}

\begin{rem}
  (1) Let $\cM$ be a uniform affine oriented matroid. It follows from Theorem~\ref{thm:char-genpos} that $\widetilde S/O_\cM$ is Cohen--Macaulay and hence the bounded complex $X(\cB_\cM)$ is also Cohen--Macaulay (over any field $\kk$) by Corollary~\ref{cor: ring CM => B_M CM}. In the case where $\cB_\cM$ is of rank 2, our Proposition~\ref{prop:rank2} thus induces Dong's theorem (\cite{D08}), saying $X(\cB_\cM)$ is homeomorphic to a ball.

  (2) Let $\cH := \cbr{H_1,\dots ,H_n}$ be an affine hyperplane arrangement in $\RR^d$ with $\bigcap H_i = \void$. 
Assume there exists $H_{i_1},\dots ,H_{i_d}$ that intersect with a point, and the bounded complex $\cB_\cM$ of the affine oriented matroid $\cM$ induced from $\cH$ is of full rank. 
Then $X(\cB_\cM)$ is homeomorphic to the barycentric subdivision $\gD$ of the polyhedral complex consisting of the polytopes given by some of $H_1,\dots ,H_n$. 
Moreover the dimension of $\gD$ is equal to $\rank \cB_\cM = d$. The (geometric) simplicial complex $\gD$ is thus identically PL embedded in $\RR^d$. Now assume $X(\cB)$, hence $\gD$, is Cohen--Macaulay. Applying the results by Miller and Reiner \cite{MR}, 
we can deduce that $X(\cB_\cM)$ is a ball whenever $d \le 3$ and a topological manifold when $d \le 4$.
\end{rem}

Next we will study the boundary $\partial X(\cB_\cM)$ of the homology manifold $X(\cB_\cM)$. 
Applying the result due to W. J. R. Mitchell \cite{Mi}, we can show that $\partial X(\cB_\cM)$ is a homology manifold without boundary over $\ZZ$ if $\widetilde S/O_\cM$ is Cohen--Macaulay. Mitchell's theorem can be applied to more general homology manifolds of first countable. For a finite regular CW complex, the theorem reads as follows.

\begin{thm}[Mitchell \cite{Mi}]
Let $X$ be a finite regular CW complex. Assume $X$ is a homology $n$-manifold over $\ZZ$ with boundary. Then its boundary $\partial X$ is a homology manifold without boundary.
\end{thm}

\begin{cor}\label{partial B_M is mfd}
  If $\widetilde S/O_\cM$ is Cohen--Macaulay, then the boundary $\partial X$ of $X := X(\cB_\cM)$ is a homology manifold without boundary over $\ZZ$.
\end{cor}

In the situation of Corollary~\ref{partial B_M is mfd}, we strongly believe that $\partial X$ is a {\it homology sphere} over $\ZZ$, that is, 
\begin{equation}\label{homology sphere}
H_i(\partial X, \partial X \setminus \{x\};\ZZ)= H_i(\partial X;\ZZ)=\begin{cases}
\ZZ & \text{if $i= \dim \partial X$,}\\
0 & \text{otherwise}
\end{cases}
\end{equation}
for all $x \in \partial X$. 
However, we have no proof yet. 

On the other hand, we have the following.

\begin{thm}\label{Delta_M ball}
Assume that $\wS/O_\cM$ is Cohen--Macaulay. 
For  the simplicial complex  $\Delta_\cM \subset 2^{\br {2n}}$  whose Stanley--Reisner ring  is $\wS/O_\cM$, $\Delta_\cM$ is a homology manifold over $\ZZ$ with boundary. Moreover, the boundary is a homology sphere in the sense of \eqref{homology sphere}. 
\end{thm}

\begin{proof}
Set $d:= \dim \Delta_\cM$. 
For the first assertion, it suffices to show that $\Delta_\cM$ is a homology manifold over $\kk$ by Lemma~\ref{lem:redhom-lem}. 
Moreover, since $\Delta_\cM$ is Cohen--Macaulay, it suffices to show that $\dim_\kk H^{d- \# F} (\opn{lk}_{\Delta_\cM} F; \kk) \le 1$ for all $\emptyset \ne F \in \Delta_\cM$, where $\opn{lk}_{\Delta_\cM} F$ denotes the link of $F$.
Identifying $F \in \Delta_\cM$ with the corresponding squarefree vector in $\ZZ^{2n}$, we have 
$$\dim_\kk H^{d - \# F} (\opn{lk}_{\Delta_\cM} F; \kk)= \dim_\kk [\omega_{\wS/O_\cM}]_F = \dim_\kk [J_\cM]_F \le 1.$$
Here the first (resp. second) equality follows from Hochster's formula (resp. Theorem~\ref{canonical}). 

For the second assertion, note that $F \in \Delta_\cM$ belongs to the boundary $\partial \Delta_\cM$, if and only if $H^{d- \# F} (\opn{lk}_{\Delta_\cM} F; \kk) = 0$, if and only if $\sfm_F \not \in J_\cM$, where $\sfm_F \in \kk[\gD_\cM]$ is the squarefree monomial corresponding to $F$, since $H^{d - \# F} (\opn{lk}_{\Delta_\cM} F; \kk) \cong  [J_\cM]_F$.  
Hence the Stanley--Reisner ring $\kk[\partial \Delta_\cM]$ of $\partial \Delta_\cM$ is isomorphic to $\kk[\Delta_\cM]/J_\cM$. 
Since $\kk[\Delta_\cM]$ is Cohen--Macaulay and $J_\cM$ is its canonical ideal in the graded context, $\kk[\partial \Delta_\cM]$ is a Gorenstein ring with $\kk[\partial\Delta_\cM] \cong \omega_{\kk[\partial\Delta_\cM]}$ as graded modules by the graded version of \cite[Proposition~3.3.18]{BH}. This means that $\partial\Delta_\cM$ is a homology sphere over $\kk$, and hence over $\ZZ$ by Lemma~\ref{lem:redhom-lem}.
\end{proof}

However, in the situation of Theorem~\ref{Delta_M ball}, we have no idea whether  the geometric realization $|\Delta_\cM|$ is homeomorphic to a ball, or even whether it is contractible. 

\section*{Acknowledgments}
The authors are very grateful to Satoshi Murai for stimulating discussion and bringing the paper \cite{Mi} to our attention. We would also like to express our thanks to Michihisa Wakui and anonymous referees for valuable comments on the presentation of the paper.


%
\end{document}